\documentclass[11pt,reqno]{amsart}

\usepackage[margin=1in]{geometry}
\usepackage{amsmath,amssymb,amsthm}
\usepackage[T1]{fontenc}
\usepackage[utf8]{inputenc}
\usepackage{enumitem}
\usepackage{xcolor}
\usepackage{hyperref}
\hypersetup{
    colorlinks=true,
    linkcolor=blue!50!black,  
    citecolor=teal,           
    urlcolor=red!50!black     
}

\numberwithin{equation}{section}

\theoremstyle{plain}
\newtheorem{theorem}{Theorem}[section]
\newtheorem{proposition}[theorem]{Proposition}
\newtheorem{lemma}[theorem]{Lemma}
\newtheorem{corollary}[theorem]{Corollary}
\theoremstyle{definition}
\newtheorem{definition}[theorem]{Definition}
\newtheorem{example}[theorem]{Example}
\theoremstyle{remark}
\newtheorem{remark}[theorem]{Remark}

\newcommand{\R}{\mathbb{R}}
\newcommand{\Mm}{\mathcal{M}}                    
\newcommand{\Nn}{\mathcal{N}}                    
\newcommand{\Sym}{S^{2}}                         
\newcommand{\tr}{\operatorname{tr}}
\newcommand{\vgrad}{\operatorname{grad}}
\newcommand{\supp}{\operatorname{supp}}
\newcommand{\Diff}{\operatorname{Diff}}

\newcommand{\LD}[1]{\mathrm{L}_{#1}}             
\newcommand{\ip}[3]{\langle #1,#2\rangle_{#3}}   
\newcommand{\dv}{\mu}                            
\newcommand{\Ee}{\mathsf{E}}                     
\newcommand{\Ham}{\mathcal{H}}                   

\begin{document}

\title[Variational calculus and the Einstein evolution equations]
{A global variational calculus on Fr\'echet manifolds and the
Lagrangian structure of the Einstein evolution equations}

\author{Jos\'e Antonio Vallejo}
\address{Departamento de Matemáticas Fundamentales, Universidad Nacional de Educación a Distancia,
Spain}
\email{jvallejo@mat.uned.es}

\subjclass[2020]{Primary 58E30, 58D17; Secondary 58B20, 83C05, 70H33}
\keywords{Fr\'echet manifolds, manifold of Riemannian metrics, DeWitt metric,
Euler--Lagrange equations, Noether theorem, momentum map, Einstein equations,
Maupertuis--Jacobi principle}

\begin{abstract} 
We develop a variational calculus for curves on Fr\'echet manifolds and apply it to the 
Lagrangian formulation of the Einstein evolution equations on the manifold $\Mm$ of 
Riemannian metrics of a compact manifold. For mechanical Lagrangians associated with weak 
pseudo-Riemannian metrics, we establish the Euler--Lagrange equations, energy 
conservation, and an infinitesimal Noether theorem that does not require the symmetry 
vector field to generate a flow. Applied to the DeWitt Lagrangian, this framework 
identifies the momentum constraint with the vanishing of a Noether momentum map and gives 
a direct proof of its propagation. We also derive the pointwise transport identity 
\[ 
\partial_t\bigl(\Ham_\lambda\,\dv_g\bigr) = 2\alpha\,\delta_g(\delta_g^-k)\,\dv_g\,, 
\] 
which implies propagation of the Hamiltonian constraint. Finally, we prove a Maupertuis--
Jacobi theorem for weak pseudo-Riemannian metrics on Fr\'echet manifolds and show that, 
away from the zero set of the total scalar-curvature potential, constrained Einstein 
evolutions are reparametrized unit-speed geodesics of the conformal DeWitt metric 
$-4\alpha S_\lambda G^-$. 
\end{abstract}

\maketitle
\tableofcontents

\section{Introduction}

Since the work of Dirac, Wheeler and DeWitt \cite{DeW67,Whe64}, and its
geometric formulation by Arnowitt, Deser and Misner \cite{ADM62}
and by Fischer and Marsden \cite{FM72a,FM72b}, it is well known that the
vacuum Einstein equations of General Relativity can be read as the evolution
equations of a dynamical system whose configuration space is the space $\Mm$
of Riemannian metrics on a compact $n-$manifold $M$ (physically, $n=3$). Under this
perspecive a spacetime is not a static object, but the history of a
Riemannian geometry, that is, a curve
\[
g\colon I\subset\R\longrightarrow\Mm,\qquad t\longmapsto g_t ,
\]
in the same way as a mechanical system is a curve in its configuration manifold. The analogy is not merely formal, it can be made quantitative. Indeed, DeWitt
\cite{DeW67} discovered that, for zero shift and unit lapse, the evolution
equations are those of a point particle moving on $\Mm$ under the potential
$-2S(g)$ (with $S$ the total scalar curvature), the kinetic term being
determined by a distinguished weak pseudo-Riemannian metric $G^{-}$ on $\Mm$,
nowadays called the \emph{DeWitt metric}. Fischer and Marsden \cite{FM72b}
developed the corresponding Hamiltonian theory, and
Gil-Medrano \cite{GM96} showed later that the picture does not depend on the
Lorentzian character of the spacetime metric and, what is more surprising,
that inside the natural one-parameter family $G^{c}$ of metrics on $\Mm$
introduced in \cite{GMi91}, the DeWitt metric is the only one for which
such a description is possible.

In a previous work \cite{Val09} the complementary
Lagrangian version was started, which consists in developing a calculus of
variations for curves in Fr\'echet manifolds (Euler--Lagrange equations,
Noether theorems, conservation laws) general enough to contain both classical
particle mechanics and geometrodynamics as literal instances, in such a way
that statements like ``the constraints are conserved'' become theorems of an
infinite dimensional analytical mechanics, and not \emph{ad hoc} computations.
The purpose of the present paper is threefold: 
\begin{enumerate}[label=\textit{\alph*)}] 

\item to develop a self-contained variational framework for curves on Fr\'echet manifolds, 
including the Euler--Lagrange equations, energy conservation, and an infinitesimal Noether 
theorem for mechanical Lagrangians associated with weak pseudo-Riemannian metrics (thus
expanding the work done in \cite{Val09});
 
\item to apply this framework to the DeWitt Lagrangian on the manifold of Riemannian 
metrics, interpreting the momentum constraint as the vanishing of a Noether momentum map 
and deriving a pointwise transport law for the Hamiltonian density, from which the 
propagation of both Einstein constraints follows; 

\item to establish a Maupertuis--Jacobi-type theorem in the weak pseudo-Riemannian Fr\'echet 
setting and to use it to show that, away from the zero set of the total scalar-curvature 
potential, constrained Einstein evolutions are reparametrized geodesics of an explicit 
conformal rescaling of the DeWitt metric. 
\end{enumerate}

\begin{remark}
Throughout the paper the term ``global'' refers
to the coordinate free (chart independent) formulation, and to the single
global chart available on the open cone $\Mm$. It does not refer to the
existence of solutions global in time, nor to a global quotient (superspace)
construction, none of which is claimed here.
\end{remark}

\subsection{Summary of results}\label{ssec:summary}
Sections~\ref{sec:setting}--\ref{sec:noether} contain the general theory,
while Sections~\ref{sec:einstein}--\ref{sec:examples} are devoted to the
applications to relativity. Let us describe the contents in some more detail.

\emph{Analysis} (Section~\ref{sec:varcalc}). After fixing the setting of
manifolds of mappings in the sense of Kriegl and Michor \cite{KM97}
(Section~\ref{sec:setting}), we prove a DuBois--Reymond lemma for weak-$*$
continuous curves in the dual of a Fr\'echet space
(Proposition~\ref{prop:DBR}). The proof given here reduces the statement to 
the scalar case, and in this way it avoids completely the vector valued integration used
in \cite{Val09}. We
then obtain the Euler--Lagrange equations in charts (Theorem~\ref{thm:EL})
and, for Lagrangians of \emph{mechanical type}
$\mathcal{L}=\tfrac12G(v,v)-V\circ\pi$ built from a weak pseudo-Riemannian
metric $G$ admitting a Levi-Civita connection and a potential admitting a
$G-$gradient, in the global form
\begin{equation}\label{eq:mechanicalEL-intro}
\nabla_{\partial_t}c'=-\vgrad V(c)
\end{equation}
(Theorem~\ref{thm:mechEL}). A key technical tool, which can be of use in other
situations, is a differentiation lemma for pairings $t\mapsto \ell(t)(w(t))$
between weak-$*$ $C^{1}$ curves of functionals and $C^{1}$ curves of vectors
(Lemma~\ref{lem:pairing}), whose proof rests on the Banach--Steinhaus theorem
for barrelled spaces. From it, the conservation of the energy
$\Ee_{\mathcal L}(v)=F\mathcal{L}(v)(v)-\mathcal{L}(v)$ along extremals
follows directly (Theorem~\ref{thm:energy}).

\emph{Symmetry} (Section~\ref{sec:noether}). We prove a Noether theorem
(Theorem~\ref{thm:noether}) under an infinitesimal invariance hypothesis
only, allowing in addition invariance up to the total derivative of a
function of position.
On a Fr\'echet manifold a vector field need not have a flow, so the hypothesis
of invariance under a group is too restrictive. For the
action of $\Diff(M)$ on $(\Mm,G^{c})$ the resulting conserved quantities
assemble into an equivariant Noether momentum map
\[
\mathbf{J}\colon T\Mm\longrightarrow \mathfrak{X}(M)',\qquad
\langle \mathbf{J}(g,k),X\rangle=G^{c}_{g}(k,\LD{X}g)
 =2\int_M g\bigl(X,(\delta^{c}_{g}k)^{\sharp}\bigr)\dv_g ,
\]
with values in the continuous dual $\mathfrak{X}(M)'$, where
$\delta^{c}_{g}$ is the $G^{c}-$adjoint of the symmetrized covariant
derivative (Proposition~\ref{prop:momentum},
Proposition~\ref{prop:equivariance}).

\emph{Relativity} (Section~\ref{sec:einstein}). Let $\alpha\neq0$ be a real
constant, $\lambda\in\R$, and consider on $\widetilde M=I\times M$ the
metric $\widetilde g=\alpha\,dt^2+g_t$ determined by a curve $g_t$ in
$\Mm$ (zero shift, constant lapse; $\alpha=-1$ is the Lorentzian vacuum
case). By a theorem of Gil-Medrano \cite{GM96}, $\widetilde g$ is Einstein
with constant $\lambda$ if and only if $g_t$ is an extremal of the
mechanical Lagrangian
\begin{equation}\label{eq:deWittLag-intro}
\mathcal{L}_{\alpha,\lambda}(g,k)
=\tfrac12\,G^{-}_{g}(k,k)-2\alpha\,S_\lambda(g),
\quad
S_\lambda(g)=\int_M\bigl(\tau(g)-\lambda(n-1)\bigr)\dv_g ,
\end{equation}
subject to the two constraints $\delta^{-}_{g}g'=0$ (momentum) and
$\Ham_\lambda(g,g')=0$ (Hamiltonian), where
$\Ham_\lambda(g,k)=\tfrac12\ip{k}{k}{g}^{-}+2\alpha\tau_\lambda(g)$ and
$\tau_\lambda=\tau-\lambda(n-1)$. Our contributions are then:

\begin{enumerate}[label=\textit{\alph*)}]
\item The identification of both constraints with Noether data. The momentum
constraint is the vanishing of the $\Diff(M)-$Noether momentum map, while the
Hamiltonian constraint is the pointwise vanishing of the \emph{integrand}
$\Ham_\lambda$ of the Noether energy, because
$\Ee_{\mathcal{L}_{\alpha,\lambda}}(g,k)=\int_M\Ham_\lambda(g,k)\,\dv_g$
(Proposition~\ref{prop:chargeident}). Thus, the Hamiltonian constraint is a
pointwise refinement of the vanishing of a single (integrated) Noether charge,
and not a Noether charge density by itself.
\item \emph{Propagation of the momentum constraint}
(Theorem~\ref{thm:momprop}). If an extremal verifies
$\delta^{-}_{g_{t_0}}g'_{t_0}=0$ at a single instant, then it verifies it for
all time. The proof follows directly from the Noether theorem and,
contrary to the usual argument through linear first order systems
\cite[Prop.~3.5]{FM72b}, it does not require any uniqueness theory for partial
differential equations (which is a delicate issue in this setting).
\item \emph{A pointwise transport law} (Theorem~\ref{thm:transport}). Along
\emph{any} extremal of $\mathcal{L}_{\alpha,\lambda}$,
\begin{equation}\label{eq:transport-intro}
\partial_t \Ham_\lambda
=-\tfrac12\tr_g(k)\,\Ham_\lambda+2\alpha\,\delta_g(\delta^{-}_{g}k),
\end{equation}
equivalently,
\begin{equation}
\partial_t\bigl(\Ham_\lambda\,\dv_g\bigr)=2\alpha\,\delta_g(\delta^{-}_{g}k)\,\dv_g .
\end{equation}
As a consequence, on those extremals verifying the momentum constraint the
density $\Ham_\lambda\,\dv_{g_t}$ is pointwise constant in time
(Corollary~\ref{cor:pointwise}); in particular the Hamiltonian constraint
propagates (Corollary~\ref{cor:hamprop}), and integrating
\eqref{eq:transport-intro} over $M$ we recover the Noether conservation of
energy. This makes precise, in the Lagrangian language, the
Fischer--Marsden continuity equation \cite[Prop.~3.4]{FM72b} and Misner's
observation on identically vanishing Hamiltonians \cite{Mis57,FM72b}.
\end{enumerate}

\emph{Geometrization} (Section~\ref{sec:jacobi}). Here we prove two results.
The first one is a conformal rescaling lemma for weak pseudo-Riemannian
metrics on Fr\'echet manifolds (Lemma~\ref{lem:conformal}), which says that if
$G$ has a Levi-Civita connection and $f$ is a nowhere vanishing function
admitting a $G-$gradient, then $fG$ has also a Levi-Civita connection, given
by the classical formula. The second one is a Maupertuis--Jacobi theorem
(Theorem~\ref{thm:jacobi}), stating that the extremals of
$\tfrac12G(v,v)-V\circ\pi$ with energy $e$ contained in $\{V\neq e\}$
correspond, through the explicit reparametrization $ds/dt=2(e-V)$, to the unit
speed geodesics of the conformal metric $\widetilde G=2(e-V)\,G$, and
conversely. Combining these two results with those of
Section~\ref{sec:einstein}, we obtain (Theorem~\ref{thm:GRjacobi}) that every
solution of the Einstein evolution system with its constraints, whose curve of
metrics stays in $\{S_\lambda\neq0\}$, is, up to the reparametrization
$ds/dt=-4\alpha S_\lambda(g_t)$, a unit speed geodesic of the conformal DeWitt
metric
\[
\mathcal{G}:=-4\alpha\,S_\lambda(g)\,G^{-} .
\]
For $\alpha=-1$, $\lambda=0$ this amounts to saying that the vacuum
gravitational histories are geodesics of $4S(g)\,G^{-}$. The correspondence
degenerates where $S_\lambda$ vanishes, in particular at every moment of time
symmetry, as we will discuss.

\emph{Examples} (Section~\ref{sec:examples}). For homothetic curves
$g_t=c(t)\,g_0$ through an Einstein metric $g_0$, the whole system reduces to
the \emph{linear} scalar equation $c''=\tfrac{A}{2}-Bc$, with the constraint
playing the role of the first integral of the energy
(Proposition~\ref{prop:homothetic}). The Milne, de Sitter and round $S^{4}$
geometries appear as the three basic solutions, and they illustrate in
particular the independence of the formalism with respect to the signature.

\subsection{Notation and conventions}\label{ssec:notation}
Throughout the paper, $M$ is a compact connected smooth manifold without
boundary, $\dim M=n\ge2$, and $I\subset\R$ is an open interval. We write
$\Sym(M):=\Gamma^{\infty}(\Sym T^{*}M)$ for the Fr\'echet space of smooth
symmetric $(0,2)-$tensor fields with the $C^{\infty}-$topology, and
\[
\Mm:=\{g\in \Sym(M): g \text{ positive definite}\}\subset \Sym(M),
\]
which is an open convex cone, hence a Fr\'echet manifold with the single
global chart given by the inclusion, and with $T_g\Mm=\Sym(M)$ for every $g$
\cite{Ebi70,GMi91,KM97}. For recent results on the geometry of this space, 
the reader can consult \cite{BHM13,Cla10,Cla13}, the older reference \cite{FG89}
is also useful.

For $g\in\Mm$, we denote by $\nabla^{g}$ the Levi-Civita connection, by
$\rho(g)$ the Ricci tensor, by $\tau(g)=\tr_g\rho(g)$ the scalar curvature and
by $\dv_g$ the Riemannian volume density. The pointwise product of
$(0,2)-$tensors is $\ip{h}{k}{g}=\tr(g^{-1}hg^{-1}k)$ and
$\tr_g h=\tr(g^{-1}h)$, while the trace free part of $h$ is
$h_0=h-\frac1n(\tr_gh)g$. For the Laplacian on functions we take the
convention $\Delta_g f=-\tr_g(\nabla^g df)$, so that
$\Delta_g=\delta_g d$, with the divergence given by
\[
(\delta_g h)_{j}=-g^{ik}\nabla_i h_{kj}
\quad(h\in \Sym(M)),\qquad
\delta_g\omega=-g^{ij}\nabla_i\omega_j\quad(\omega\in\Omega^1(M)),
\]
and extended to covariant tensors by the same formula (thus, $\delta_g$ lowers
the covariant order by one). Note that $\delta_g(fg)=-df$ for
$f\in C^{\infty}(M)$. The symmetrized covariant derivative is
$\delta^{*}_g\colon\Omega^{1}(M)\to \Sym(M)$,
$(\delta^{*}_g\omega)_{ij}=\tfrac12(\nabla_i\omega_j+\nabla_j\omega_i)$, so
that $\LD{X}g=2\delta^{*}_g(X^{\flat})$ and, on the closed manifold $M$,
\begin{equation}\label{eq:adjointness}
\int_M \ip{\delta^{*}_g\omega}{h}{g}\dv_g=\int_M\ip{\omega}{\delta_g h}{g}\dv_g,
\qquad
\int_M(\delta_g\omega)\,f\,\dv_g=\int_M\ip{\omega}{df}{g}\dv_g ,
\end{equation}
see \cite[1.59--1.60]{Bes87}. The musical isomorphisms $\flat,\sharp$ are
taken with respect to $g$ (or to $g_t$, as will be clear from the context). We
denote by $\mathfrak{X}(M)$ the Lie algebra of smooth vector fields, by
$\Diff(M)$ the diffeomorphism group, and by $\LD{X}$ the Lie derivative.
Einstein metrics on the $(n+1)-$manifold $\widetilde M$ are normalized by
$\rho(\widetilde g)=\lambda\widetilde g$, and in terms of the cosmological
constant $\Lambda$ of the field equations
$\rho(\widetilde g)-\tfrac12\tau(\widetilde g)\widetilde g
+\Lambda\widetilde g=0$ one has $\Lambda=\tfrac{n-1}{2}\lambda$.

Throughout, smoothness in infinite dimensions is understood in the convenient sense of Kriegl and Michor \cite{KM97}. For maps between open subsets of Fr\'echet spaces, convenient smoothness coincides with Michal--Bastiani smoothness, the differential calculus used by Hamilton in \cite{Ham82}; see \cite[Ch.~I and Thm.~4.11]{KM97}. We denote by $E'$ the continuous dual of a Fr\'echet space $E$. Since every Fr\'echet space is bornological, every bounded linear functional on $E$ is continuous. Consequently, the differential of a conveniently smooth real-valued map on an open subset of $E$ is an element of $E'$.

\section{The manifold of Riemannian metrics and the DeWitt family}
\label{sec:setting}

\subsection{Manifolds modeled on Fr\'echet spaces}\label{ssec:frechet}
Let us briefly recall the facts we will need, referring the reader to
\cite{KM97} for the proofs (see also \cite{Ham82,Mic20}). If $N$ is a finite
dimensional manifold and $M$ is compact, the space $C^{\infty}(M,N)$ is a
smooth manifold modeled on the Fr\'echet spaces $\Gamma^{\infty}(f^{*}TN)$,
with charts built from the Riemannian exponential map of any auxiliary metric
on $N$ (a \emph{local addition}). The resulting structure does not depend on
these choices, and $T_fC^{\infty}(M,N)\cong\Gamma^{\infty}(f^{*}TN)$
\cite[\S42]{KM97}. Open subsets such as $\operatorname{Imm}(M,N)$,
$\operatorname{Emb}(M,N)$ and $\Diff(M)$ inherit the structure, and $\Diff(M)$
is a Fr\'echet Lie group whose Lie algebra is $\mathfrak{X}(M)$ (with a caveat
about a sign in the bracket, which will play no role here), see
\cite[\S43]{KM97}.

For our purposes, the essential example is $\Mm\subset \Sym(M)$, which, as we
have already noted in \S\ref{ssec:notation}, carries a single global chart.
Thus, the calculus on $\Mm$ is literally the calculus on an open subset of a
Fr\'echet space, and we will make a systematic use of this fact. All the
abstract results in Sections~\ref{sec:varcalc}--\ref{sec:jacobi} are therefore
stated for an open set $\Nn$ of a Fr\'echet space $E$ (``the global chart
setting''). They extend word by word to manifolds admitting a local addition,
by working with connectors and chart representatives as in \cite{Val09},
but as $\Mm$ does not need such an extension, we will not insist on this
point.

Two facts about smooth curves will be used constantly \cite[Ch.~I]{KM97}. The
first one is that a map into a Fr\'echet space is smooth if and only if it is
smooth along smooth curves. The second one is that, for a smooth map
$L\colon U\subset E\to\R$ and a smooth curve $c\colon\R\to U$, the chain rule
$\frac{d}{ds}L(c(s))=dL(c(s))(c'(s))$ holds, where $dL(u)\in E'$ is the
(G\^ateaux) differential. Partial derivatives $D_1L,D_2L$ of maps
$L\colon U\times E\to\R$ are defined in the obvious way, and are again
continuous linear functionals in view of the bornological property of $E$.

\subsection{The metrics $G^{c}$ on $\Mm$}\label{ssec:Gc}
Each $g\in\Mm$ splits $T_g\Mm=\Sym(M)$ algebraically as
$V_0(g)\oplus V_1(g)$, where $V_0(g)=\{h:\tr_gh=0\}$ and
$V_1(g)=C^{\infty}(M)\,g$. Following \cite{GMi91,GM96}, for each real
$c\neq0$ we define the pointwise product
\begin{equation}\label{eq:cproduct}
\ip{h}{k}{g}^{c}:=\tr(g^{-1}h_0g^{-1}k_0)+c\,\tr_g(h)\,\tr_g(k),
\qquad h,k\in \Sym(M),
\end{equation}
and, by integration, the weak pseudo-Riemannian metric on $\Mm$
\[
G^{c}_{g}(h,k):=\int_M \ip{h}{k}{g}^{c}\,\dv_g .
\]
For $c=\frac1n$ one recovers the usual $L^2$ (Ebin) metric
$G_g(h,k)=\int_M\ip{h}{k}{g}\dv_g$ \cite{Ebi70,GMi91}. For $c>0$, $G^{c}$
is a weak Riemannian metric, whereas for $c<0$ it is positive definite on
$V_0(g)$ and negative definite on $V_1(g)$. The distinguished value
\[
c^{-}:=\frac{1-n}{n}\qquad(\text{note } c^{-}<0 \text{ for } n\ge2)
\]
gives the \emph{DeWitt metric} $G^{-}:=G^{c^{-}}$ \cite{DeW67,GM96}.

\begin{lemma}\label{lem:cidentities}
For all $h,k\in \Sym(M)$ and $g\in\Mm$:
\begin{enumerate}[label=\textit{\alph*)}]
\item\label{aux1} $\ip{h}{k}{g}^{c}=\ip{h}{k}{g}+\bigl(c-\tfrac1n\bigr)\tr_g(h)\tr_g(k)$;
in particular
\[
\ip{h}{k}{g}^{-}=\tr(g^{-1}hg^{-1}k)-\tr_g(h)\,\tr_g(k).
\]
\item\label{aux2} Each $\ip{\cdot}{\cdot}{g}^{c}$ ($c\neq0$) is a (pointwise)
non degenerate symmetric bilinear form on $\Sym_p(T^{*}M)$ for every
$p\in M$; consequently $G^{c}$ is \emph{weakly non degenerate}, that is, if
$G^{c}_g(h,k)=0$ for all $k\in \Sym(M)$ then $h=0$.
\end{enumerate}
\end{lemma}

\begin{proof}\mbox{}

\ref{aux1} Insert $h=h_0+\frac1n(\tr_gh)g$, $k=k_0+\frac1n(\tr_gk)g$ into
$\ip{h}{k}{g}$ and use $\tr_g(h_0)=0$, $\ip{g}{g}{g}=n$, to get
$\ip{h}{k}{g}=\ip{h_0}{k_0}{g}+\frac1n\tr_g(h)\tr_g(k)$. Subtracting this from
\eqref{eq:cproduct} gives the identity, and for $c=c^{-}$ we have
$c^{-}-\frac1n=-1$.

\ref{aux2} Pointwise, if $\ip{h}{k}{g}^{c}=0$ 
for all $k$ at $p$, testing with
$k=h_0$ gives $|h_0|_g^2(p)=0$ and testing with $k=g$ gives
$cn\,\tr_gh(p)=0$, so $h(p)=0$. As for the global statement, suppose that
$h(p_0)\neq0$ at some $p_0\in M$. By pointwise non degeneracy there is
$k_{p_0}\in \Sym_{p_0}$ with $\ip{h}{k_{p_0}}{g}^{c}(p_0)\neq0$ and,
replacing $k_{p_0}$ by $-k_{p_0}$ if necessary, we may take
$\ip{h}{k_{p_0}}{g}^{c}(p_0)>0$. Extend $k_{p_0}$ to $k\in \Sym(M)$. Then
$\ip{h}{k}{g}^{c}>0$ on a neighborhood of $p_0$ and, multiplying $k$ by a
nonnegative bump function supported there, we get
$\int_M\ip{h}{k}{g}^{c}\dv_g>0$, so $G^{c}_g(h,k)\neq0$.
\end{proof}

\subsection{Geodesics, gradients and divergences}\label{ssec:christoffel}
The Levi-Civita connection of $G^{c}$ exists, and was computed in
\cite{GMi91} (on the larger space of all non degenerate $(0,2)-$tensor
fields, of which $\Mm$ is a geodesically closed submanifold; see also
\cite[Prop.~3.1]{GM96}). Connections will be encoded, in the global chart, by
their Christoffel map, with the sign convention of \cite{GMi91,GM96}.
This is a smooth map $\Gamma\colon \Nn\to L^{2}_{\mathrm{sym}}(E;E)$,
$g\mapsto\Gamma_g$, where $L^{2}_{\mathrm{sym}}(E;F)$ denotes the space of
continuous symmetric bilinear maps $E\times E\to F$ (smoothness of
$\Nn$-valued maps into these spaces being understood, as everywhere here, in
the convenient calculus of \cite{KM97} via joint evaluation), which defines
the covariant derivative
\begin{equation}\label{eq:covder}
\nabla_{\partial_t}h:=h'-\Gamma_{c(t)}(c'(t),h(t))
\end{equation}
of a vector field $h(t)$ along a curve $c(t)$ (and similarly along smooth
maps of several variables), so that the geodesic equation reads
$c''=\Gamma_c(c',c')$. In this language, being free of torsion amounts to the
symmetry of $\Gamma_g$, and the metricity of $\nabla$ for a weak metric $G$ is
the identity (involving the Fréchet differential $D$)
\begin{equation}\label{eq:metricity}
DG(g)(l)(h,k)+G_g\bigl(\Gamma_g(l,h),k\bigr)+G_g\bigl(h,\Gamma_g(l,k)\bigr)=0,
\qquad l,h,k\in E,
\end{equation}
as one checks by evaluating $X\,G(Y,Z)=G(\nabla_XY,Z)+G(Y,\nabla_XZ)$ on
constant fields. When a symmetric $\Gamma$ verifying \eqref{eq:metricity}
exists it is unique (Lemma \ref{lem:LCunique} below), and we call $\nabla$
\emph{the} Levi-Civita connection of $G$.

For $G^{c}$ the Christoffel map is, by \cite{GMi91,GM96},
\begin{equation}\label{eq:spray}
\Gamma^{c}_{g}(h,h)=hg^{-1}h-\tfrac12\tr_g(h)\,h+\tfrac{1}{4cn}\ip{h}{h}{g}^{c}\,g,
\end{equation}
(extended to $\Gamma^{c}_g(h,k)$ by polarization), so that for the DeWitt
value $c^{-}=\frac{1-n}{n}$, that is, $cn=1-n$,
\begin{equation}\label{eq:sprayminus}
\Gamma^{-}_{g}(h,h)=hg^{-1}h-\tfrac12\tr_g(h)\,h-\tfrac{1}{4(n-1)}\ip{h}{h}{g}^{-}\,g .
\end{equation}
For $n=3$ this reads
$\Gamma^{-}_g(h,h)=hg^{-1}h-\tfrac12\tr_g(h)h-\tfrac18\ip{h}{h}{g}^{-}g$.

A \emph{Riemannian functional} is a smooth $F\colon\Mm\to\R$ such that
$F(\varphi^{*}g)=F(g)$ for all $\varphi\in\Diff(M)$ \cite[4.1]{Bes87}. We
will use the volume $V(g)=\int_M\dv_g$, the total scalar curvature
$S(g)=\int_M\tau(g)\dv_g$ and, for $\lambda\in\R$,
\begin{equation}\label{eq:Slambda}
S_\lambda(g):=\int_M\tau_\lambda(g)\,\dv_g,\qquad
\tau_\lambda(g):=\tau(g)-\lambda(n-1),
\end{equation}
so that $S_\lambda=S-\lambda(n-1)V$. A functional $F$ \emph{admits a
$G^{c}-$gradient} if there exists a smooth
$\vgrad^{c}F\colon\Mm\to \Sym(M)$ with
$dF(g)(h)=G^{c}_g(\vgrad^{c}F(g),h)$ for all $h$, and by weak non degeneracy
it is then unique. From the first variation formulas
\begin{equation}\label{eq:variations}
D\dv(g)(h)=\tfrac12\tr_g(h)\,\dv_g,
\quad
D\tau(g)(h)=\Delta_g(\tr_gh)+\delta_g\delta_g h-\ip{\rho(g)}{h}{g}
\end{equation}
\cite[1.174]{Bes87} one obtains \cite[Lemmas 3.2, 3.3]{GM96}:
\begin{equation}\label{eq:gradients}
\vgrad^{c}V(g)=\frac{1}{2cn}\,g,
\quad
\vgrad^{c}S(g)=-\rho_0(g)+\frac{n-2}{2cn^{2}}\,\tau(g)\,g,
\end{equation}
and hence, for the DeWitt metric ($cn=1-n$, $cn^{2}=n(1-n)$),
\begin{equation}\label{eq:gradminus}
\vgrad^{-}V(g)=-\frac{1}{2(n-1)}\,g,
\quad
\vgrad^{-}S(g)=-\rho(g)+\frac{1}{2(n-1)}\,\tau(g)\,g,
\end{equation}
\begin{align}\label{eq:gradSlambda}
\vgrad^{-}S_\lambda(g)
=& \vgrad^{-}S(g)-\lambda(n-1)\,\vgrad^{-}V(g) \nonumber\\
=& -\rho(g)+\frac{1}{2(n-1)}\,\tau(g)\,g+\frac{\lambda}{2}\,g .
\end{align}
(The middle formula in \eqref{eq:gradminus} follows from
\eqref{eq:gradients} by
$-\rho_0+\frac{n-2}{2n(1-n)}\tau g
=-\rho+\bigl(\frac1n-\frac{n-2}{2n(n-1)}\bigr)\tau g
=-\rho+\frac{1}{2(n-1)}\tau g$.)

Finally, the \emph{$G^{c}-$divergence} $\delta^{c}_g$ is the formal
$G^{c}-$adjoint of $\delta^{*}_g$, that is,
$G^{c}_g(\delta^{*}_g\omega,h)=\int_M\ip{\omega}{\delta^{c}_gh}{g}\dv_g$
for all $\omega\in\Omega^{1}(M)$.

\begin{lemma}\label{lem:deltac}
$\delta^{c}_{g}h=\delta_g(h_0)-c\,d(\tr_gh)
=\delta_gh+\bigl(\tfrac1n-c\bigr)d(\tr_gh)$; in particular
\begin{equation}\label{eq:deltaminus}
\delta^{-}_{g}h=\delta_gh+d(\tr_gh)=\delta_g\bigl(h-(\tr_gh)\,g\bigr).
\end{equation}
Moreover, we have $G^{c}_g(k,\LD{X}g)=2\int_M\ip{X^{\flat}}{\delta^{c}_gk}{g}\dv_g$
for every $X\in\mathfrak{X}(M)$, and
$\delta^{c}_g\bigl(\vgrad^{c}F(g)\bigr)=0$ for every Riemannian functional
$F$ admitting a $G^{c}-$gradient.
\end{lemma}

\begin{proof}
Write $\ip{a}{h}{g}^{c}=\ip{a}{h_0}{g}+c\,\tr_g(a)\tr_g(h)$ (by
Lemma~\ref{lem:cidentities}\,(a) along with $\ip{a}{h_0}{g}=\ip{a_0}{h_0}{g}$).
With $a=\delta^{*}_g\omega$ we have $\tr_g(\delta^{*}_g\omega)
=g^{ij}\nabla_i\omega_j=-\delta_g\omega$, so by \eqref{eq:adjointness}
\begin{align*}
G^{c}_g(\delta^{*}_g\omega,h)
=& \int_M\Bigl[\ip{\omega}{\delta_g(h_0)}{g}-c\,(\delta_g\omega)\,\tr_gh\Bigr]\dv_g\\
=&\int_M\ip{\omega}{\,\delta_g(h_0)-c\,d(\tr_gh)}{g}\,\dv_g .
\end{align*}
This proves the formula. The second expression follows from
$\delta_g(h_0)=\delta_gh+\frac1n d(\tr_gh)$ (use $\delta_g(fg)=-df$), and
\eqref{eq:deltaminus} from $\frac1n-c^{-}=1$. The displayed identity for
$\LD{X}g=2\delta^{*}_g(X^{\flat})$ is immediate. The last claim is
\cite[Prop.~3.4 and 2.6]{GM96}. Indeed, differentiating
$F(\varphi_s^{*}g)=F(g)$ along the flow $\varphi_s$ of $X$ gives
$0=dF(g)(\LD{X}g)=G^{c}_g(\vgrad^{c}F(g),2\delta^{*}_gX^{\flat})
=2\int\ip{X^{\flat}}{\delta^{c}_g\vgrad^{c}F(g)}{g}\dv_g$ for all $X$ and,
since $g$ is a positive definite metric on one-forms and $X$ (hence
$X^{\flat}$) is arbitrary, this forces $\delta^{c}_g\vgrad^{c}F(g)=0$.
\end{proof}

\begin{remark}\label{rem:momconstraintform}
Formula \eqref{eq:deltaminus} says that, with $k=g'$, the equation
$\delta^{-}_gg'=0$ is exactly the classical ADM momentum constraint
$\nabla^{j}\bigl(k_{ij}-(\tr_gk)g_{ij}\bigr)=0$ (up to the overall sign of
$\delta_g$), cf.\ \cite{ADM62,FM72b}.
\end{remark}

\section{A variational calculus for curves in Fr\'echet manifolds}
\label{sec:varcalc}

Throughout this section and the next one, $E$ is a Fr\'echet space,
$\Nn\subset E$ is open (the configuration manifold in its global
chart), $T\Nn=\Nn\times E$ with projection $\pi$, and
$\mathcal{L}\in C^{\infty}(T\Nn,\R)$ is a \emph{Lagrangian}. We write
$L=\mathcal{L}$ for its (tautological) chart representative, a smooth map
$L\colon \Nn\times E\to\R$. Curves $c\colon I\to\Nn$ are smooth, with
$I\subset\R$ open.

\subsection{Action functionals and critical curves}\label{ssec:critical}

\begin{definition}\label{def:variation}
A \emph{variation with compact support} of $c\in C^{\infty}(I,\Nn)$ is a
smooth map $c_{\bullet}\colon\R\times I\to\Nn$, $(s,t)\mapsto c_s(t)$,
with $c_0=c$ and such that there is a compact $K\subset I$ with
$c_s(t)=c(t)$ for all $s$ and all $t\notin K$. Its \emph{infinitesimal
generator} is $A:=\partial_s|_{0}c_s\in C^{\infty}(I,E)$, which verifies
$\supp A\subset K$.
\end{definition}

\begin{lemma}\label{lem:realisation}
For every $A\in C^{\infty}(I,E)$ with compact support there is a variation
with compact support of $c$ whose generator is $A$. If $\Nn$ is in
addition convex (for instance $\Nn=\Mm$), one may take
$c_s(t)=c(t)+s\,\chi(s)A(t)$ with $\chi$ a suitable cutoff.
\end{lemma}

\begin{proof}
The set $W=\{(s,t)\in\R\times I: c(t)+sA(t)\in\Nn\}$ is open (preimage of
the open $\Nn$ under a continuous map) and contains $\{0\}\times I$. Since
$\supp A=:K$ is compact, there is $\varepsilon>0$ with
$(-\varepsilon,\varepsilon)\times K\subset W$, and outside $K$ the curve is
untouched. Choose now $\chi\in C^{\infty}(\R)$, $\chi\equiv1$ near $0$,
$\supp\chi\subset(-\varepsilon,\varepsilon)$, and put
$c_s(t)=c(t)+s\chi(s)A(t)$. This is a smooth map into $\Nn$, it equals $c$
off $K$, and $\partial_s|_0c_s=A$. (On a general manifold with local
addition $\theta$, take $c_s(t)=\theta(s\chi(s)A(t))$.)
\end{proof}

Notice that the simpler
$c_s(t)=c(t)+sA(t)$ has values in $\Nn$ only for $|s|$ small, and so it
defines a variation only after the cutoff in $s$, but this is not a problem
because only the germ at $s=0$ enters Definition~\ref{def:critical}.

\begin{definition}\label{def:critical}
The \emph{action density} of $\mathcal{L}$ is
$f_{\mathcal{L}}(c):=t\mapsto L(c(t),c'(t))$. A curve
$c\in C^{\infty}(I,\Nn)$ is \emph{critical} for $\mathcal{L}$ if for every
variation with compact support $c_\bullet$ of $c$, and some (equivalently,
any) relatively compact open $J\subset I$ containing the compact set of
Definition~\ref{def:variation},
\[
\frac{d}{ds}\Big|_{s=0}\int_{J}L\bigl(c_s(t),c_s'(t)\bigr)\,dt=0 .
\]
\end{definition}

The equivalence between ``some'' and ``any'' holds because outside $K$ the
integrand does not depend on $s$. This is the notion of critical point of
\cite{GM01,Val09}, but expressed without any reference to a manifold
structure on the space of curves. Notice that the derivative with respect to $s$
exists, because $(s,t)\mapsto L(c_s(t),c_s'(t))$ is a smooth real function of
two real variables, being the composition of the smooth curve
$(s,t)\mapsto(c_s(t),c_s'(t))\in\Nn\times E$ with the conveniently smooth $L$,
and smooth maps send smooth curves to smooth curves \cite[Ch.~I]{KM97}.

\subsection{Basic lemmas}\label{ssec:analytic}
A map $\ell\colon J\to E'$ ($J\subset\R$ an interval) is called
\emph{weak-$*$ continuous} (resp.\ \emph{weak-$*$ $C^1$}) if
$t\mapsto\ell(t)(y)$ is continuous (resp.\ $C^1$) for each $y\in E$.

\begin{lemma}\label{lem:BS}
Let $E$ be a Fr\'echet (more generally, barrelled locally convex) space
and $\ell\colon[a,b]\to E'$ weak-$*$ continuous. Then
$\{\ell(t):t\in[a,b]\}$ is equicontinuous, that is, there is a continuous
seminorm $p$ on $E$ with $|\ell(t)(y)|\le p(y)$ for all $t,y$, and
consequently $(t,y)\mapsto\ell(t)(y)$ is jointly continuous on
$[a,b]\times E$.
\end{lemma}

\begin{proof}
For each $y$, the function $t\mapsto\ell(t)(y)$ is continuous on the compact
$[a,b]$, hence bounded. Thus the family is pointwise bounded, and the
Banach--Steinhaus theorem for barrelled spaces \cite[Thm.~33.1]{Tre06} gives
the equicontinuity, that is, the seminorm $p$. As for the joint continuity,
$|\ell(t)(y)-\ell(t_0)(y_0)|
\le p(y-y_0)+|(\ell(t)-\ell(t_0))(y_0)|$.
\end{proof}

\begin{proposition}[DuBois--Reymond lemma]\label{prop:DBR}
Let $E$ be a Fr\'echet space, $[a,b]\subset\R$, and let
$f,\gamma\colon[a,b]\to E'$ be weak-$*$ continuous. The following are
equivalent:
\begin{enumerate}[label=\textit{\alph*)}]
\item\label{dbr1} $\displaystyle\int_a^b\bigl[f(t)(\mu(t))+\gamma(t)(\mu'(t))\bigr]dt=0$
for every $\mu\in C^{\infty}([a,b],E)$ with
$\supp\mu\subset(a,b)$;
\item\label{dbr2} $\gamma$ is weak-$*$ $C^{1}$ with
$\frac{d}{dt}\bigl[\gamma(t)(y)\bigr]=f(t)(y)$ for every $y\in E$,
equivalently
$\gamma(t)(y)=\gamma(a)(y)+\int_a^tf(s)(y)\,ds$ for all $t,y$.
\end{enumerate}
\end{proposition}

\begin{proof}
Let us prove first the scalar case ($E=\R$), namely, that if
$F,G\in C^{0}([a,b],\R)$ verify
$\int_a^b(F\mu+G\mu')\,dt=0$ for all $\mu\in C^{\infty}_{c}\bigl((a,b)\bigr)$,
then $G\in C^{1}$ and $G'=F$. Indeed, put $\Phi(t)=\int_a^tF$. Integrating
by parts (there are no boundary terms, as $\mu$ has compact support in the open
interval) we get $\int_a^b(G-\Phi)\mu'\,dt=0$ for all such $\mu$. Fix
$\chi\in C^{\infty}_{c}((a,b))$ with $\int\chi=1$. For an arbitrary
$\psi\in C^{\infty}_{c}((a,b))$, the function
$\mu(t):=\int_a^t\bigl[\psi(s)-\bigl(\smallint\psi\bigr)\chi(s)\bigr]ds$
is smooth, vanishes near $a$, and is constant equal to
$\smallint\psi-\smallint\psi=0$ near $b$, hence
$\mu\in C^{\infty}_{c}((a,b))$ and $\mu'=\psi-(\smallint\psi)\chi$.
Substituting,
\[
0=\int_a^b(G-\Phi)\psi\,dt-\Bigl(\int_a^b\psi\,dt\Bigr)\int_a^b(G-\Phi)\chi\,dt
 =\int_a^b\bigl[(G-\Phi)-\kappa\bigr]\psi\,dt
\]
(where $\kappa:=\int_a^b(G-\Phi)\chi\,dt$), for all $\psi$, so $G-\Phi\equiv\kappa$ by continuity, that is, $G'=F$.

\ref{dbr1}$\Rightarrow$\ref{dbr2}. Fix $y\in E$ and apply \ref{dbr1} to $\mu(t)=\varphi(t)\,y$
with $\varphi\in C^{\infty}_{c}((a,b))$ arbitrary. This yields
$\int[F_y\varphi+G_y\varphi']=0$ with $F_y(t)=f(t)(y)$,
$G_y(t)=\gamma(t)(y)$ continuous, and the scalar case gives \ref{dbr2}. Note
that only the very special test curves $\varphi\,y$ are needed here.

\ref{dbr2}$\Rightarrow$\ref{dbr1}. Given $\mu$ as in \ref{dbr1}, the function
$t\mapsto\gamma(t)(\mu(t))$ is $C^{1}$ with derivative
$f(t)(\mu(t))+\gamma(t)(\mu'(t))$ by Lemma~\ref{lem:pairing} below, so
integrating and using $\mu(a)=\mu(b)=0$ we get \ref{dbr1}. Notice that there is no
circularity, because Lemma~\ref{lem:pairing} takes hypothesis \ref{dbr2} as
given, and does not use the present proposition.
\end{proof}

\begin{lemma}\label{lem:pairing}
Let $f,\gamma\colon[a,b]\to E'$ be weak-$*$ continuous and verify \ref{dbr2} of
Proposition~\textup{\ref{prop:DBR}}. Then for every
$w\in C^{1}([a,b],E)$ the function $t\mapsto\gamma(t)(w(t))$ is $C^{1}$,
with
\[
\frac{d}{dt}\bigl[\gamma(t)(w(t))\bigr]=f(t)(w(t))+\gamma(t)(w'(t)).
\]
\end{lemma}

\begin{proof}
Fix $t_0$. For $t$ near $t_0$, \ref{dbr2} gives pointwise in $y$
\[
\gamma(t)(w(t))
=\gamma(t_0)(w(t))+\int_{t_0}^{t}f(s)(w(t))\,ds .
\]
The first term is $C^{1}$ in $t$ with derivative $\gamma(t_0)(w'(t))$,
being the composition of the continuous linear $\gamma(t_0)$ with a $C^{1}$
curve. For the second one, put $\varphi(s,t)=f(s)(w(t))$. By
Lemma~\ref{lem:BS} there is a continuous seminorm $p$ with
$|f(s)(z)|\le p(z)$ uniformly in $s$, hence
$|\varphi(s,t)-\varphi(s,t')|\le p(w(t)-w(t'))$ uniformly in $s$ and,
together with the continuity in $s$ for fixed $t$, this shows that $\varphi$
is jointly continuous. The same argument applied to
$\partial_t\varphi(s,t)=f(s)(w'(t))$ shows that $\partial_t\varphi$ exists and
is jointly continuous. The classical Leibniz rule then gives
\begin{align*}
\frac{d}{dt}\int_{t_0}^{t}\varphi(s,t)\,ds
=& \varphi(t,t)+\int_{t_0}^{t}f(s)(w'(t))\,ds\\
=& f(t)(w(t))+\Bigl[\gamma(t)-\gamma(t_0)\Bigr](w'(t)),
\end{align*}
where we have used \ref{dbr2} again in the last step. Adding the two contributions
proves the formula, and the continuity of the right hand side
(Lemma~\ref{lem:BS}) gives $C^{1}$.
\end{proof}

\subsection{The Euler--Lagrange equations}\label{ssec:EL}

\begin{theorem}[Euler--Lagrange equations]\label{thm:EL}
A curve $c\in C^{\infty}(I,\Nn)$ is critical for
$\mathcal{L}\in C^{\infty}(T\Nn,\R)$ if and only if for every $y\in E$ the
function $t\mapsto D_2L(c(t),c'(t))(y)$ is $C^{1}$ on $I$ and
\begin{equation}\label{eq:EL}
\frac{d}{dt}\Bigl[D_2L\bigl(c(t),c'(t)\bigr)(y)\Bigr]
=D_1L\bigl(c(t),c'(t)\bigr)(y),
\qquad t\in I,\ y\in E .
\end{equation}
\end{theorem}

\begin{proof}
Let $c_\bullet$ be a variation with compact support $K$, generator $A$,
and let $J\supset K$ be relatively compact open with $\bar J=[a,b]\subset I$.
As we have noted after Definition~\ref{def:critical},
$(s,t)\mapsto L(c_s(t),c'_s(t))$ is a smooth function on $\R\times I$, so
by the classical rule of differentiation under the integral sign and the
chain rule along the smooth curve $s\mapsto(c_s(t),c'_s(t))$
\cite[Ch.~I]{KM97},
\begin{equation}\label{eq:firstvariation}
\frac{d}{ds}\Big|_{0}\int_{J}L(c_s,c'_s)\,dt
=\int_a^b\Bigl[D_1L(c,c')\bigl(A(t)\bigr)+D_2L(c,c')\bigl(A'(t)\bigr)\Bigr]dt,
\end{equation}
where we have used
$\partial_s|_0c'_s=\partial_t\partial_s|_0c_s=A'$ (this equality of mixed
partials for smooth real-parameter families with values in a Fr\'echet
space is valid because it holds after composing with any $\ell\in E'$, and $E'$
separates points).

Now define $f(t):=D_1L(c(t),c'(t))$ and $\gamma(t):=D_2L(c(t),c'(t))$, which are
elements of $E'$ by \S\ref{ssec:frechet} and both weak-$*$ continuous (indeed
smooth in $t$ for each fixed $y$, again by the chain rule). If $c$ is
critical, then by Lemma~\ref{lem:realisation} the right hand side of
\eqref{eq:firstvariation} vanishes for every
$\mu=A\in C^{\infty}([a,b],E)$ with support in $(a,b)$, and
Proposition~\ref{prop:DBR} gives \eqref{eq:EL} on $J$; covering $I$ by
such $J$ we get \eqref{eq:EL} on $I$. Conversely, if \eqref{eq:EL} holds,
then Lemma~\ref{lem:pairing} shows that the integrand in
\eqref{eq:firstvariation} equals $\frac{d}{dt}[\gamma(t)(A(t))]$, whose
integral vanishes since $A(a)=A(b)=0$, hence $c$ is critical.
\end{proof}

\begin{remark}\label{rem:ELremarks}
(a) Equation \eqref{eq:EL} is the statement of \cite[Thm.~5.3]{Val09}, but now
with explicit continuity hypotheses and with some analytic gaps closed. (b)
Criticality (Definition~\ref{def:critical}) is coordinate independent in an
obvious way, as it has been defined without any reference to a chart. The
chart expression \eqref{eq:EL} transforms in a compatible manner. Indeed, the
fibre derivative $F\mathcal{L}(v)\in E'$,
$F\mathcal{L}(v)(w):=\frac{d}{d\varepsilon}\big|_{0}\mathcal{L}(v+\varepsilon w)
=D_2L(\pi(v),v)(w)$, transforms as a covector, but the plain time derivative
$\frac{d}{dt}F\mathcal{L}(c')$ is not an intrinsic covector field,
because under a nonlinear change of chart it gets terms involving the
derivative of the transition map. An intrinsic formulation of the equation
therefore requires a connection, and this will be done in \S\ref{ssec:mechanical} 
by analogy with the case of particle mechanics, where \eqref{eq:EL} will become the
coordinate free \eqref{eq:mechEL}. (c) On a general Fr\'echet manifold with
local addition the same proof can be applied in each chart along the curve.
\end{remark}

\subsection{Lagrangians of mechanical type}\label{ssec:mechanical}
Let us now specialize to the class of Lagrangians which is relevant for
geometrodynamics.

\begin{definition}\label{def:mechanical}
A \emph{weak pseudo-Riemannian structure with Levi-Civita connection} on
$\Nn$ is a pair $(G,\Gamma)$ where
$G\colon\Nn\to L^{2}_{\mathrm{sym}}(E;\R)$ is smooth with each $G_g$
weakly non degenerate, and $\Gamma\colon\Nn\to L^{2}_{\mathrm{sym}}(E;E)$
is smooth and verifies the metricity identity \eqref{eq:metricity}. A
function $V\in C^{\infty}(\Nn)$ (the potential) \emph{admits a $G-$gradient} if there is a
smooth $\vgrad V\colon\Nn\to E$ with $dV(g)(h)=G_g(\vgrad V(g),h)$. The
associated \emph{mechanical Lagrangian} is
\[
\mathcal{L}(g,k)=\tfrac12\,G_g(k,k)-V(g).
\]
\end{definition}

\begin{lemma}\label{lem:LCunique}
For a given weakly non degenerate $G$ there is at most one symmetric
$\Gamma$ satisfying \eqref{eq:metricity}.
\end{lemma}

\begin{proof}
If $\Gamma,\widetilde\Gamma$ both satisfy \eqref{eq:metricity}, then $B:=\Gamma-\widetilde\Gamma$ is
symmetric and $T(l,h,k):=G_g(B(l,h),k)$ verifies $T(l,h,k)=-T(l,k,h)$
(from \eqref{eq:metricity}) and $T(l,h,k)=T(h,l,k)$ (by the symmetry of $B$).
The braid identity
$T(l,h,k)=-T(l,k,h)=-T(k,l,h)=T(k,h,l)=T(h,k,l)=-T(h,l,k)=-T(l,h,k)$
forces $T=0$, hence $B=0$ by weak non degeneracy.
\end{proof}

Let us recall the covariant derivative along curves \eqref{eq:covder} and its
obvious extension along smooth maps $a\colon\R^2\to\Nn$, given by
$\nabla_{\partial_s}(\partial_t a)=\partial_s\partial_t a
-\Gamma_a(\partial_s a,\partial_t a)$, and so on. There are two standard identities
that follow directly from the symmetry of $\Gamma$ and \eqref{eq:metricity},
\begin{align}
\nabla_{\partial_s}\partial_t a&=\nabla_{\partial_t}\partial_s a,
\label{eq:symmetrylemma}\\
\partial_t\,G_{c}\bigl(h(t),k(t)\bigr)
&=G_c\bigl(\nabla_{\partial_t}h,k\bigr)+G_c\bigl(h,\nabla_{\partial_t}k\bigr),
\label{eq:metricityalongcurves}
\end{align}
the latter for vector fields $h,k$ along a curve $c$ (expand the
left hand side by the chain rule and use \eqref{eq:metricity} with
$l=c'$).

\begin{theorem}
\label{thm:mechEL}
Let $(G,\Gamma)$ and $V$ be as in Definition~\textup{\ref{def:mechanical}}
and $\mathcal{L}=\tfrac12G(v,v)-V\circ\pi$. Then for every variation with
compact support of $c\in C^{\infty}(I,\Nn)$, with generator $A$,
\begin{equation}\label{eq:firstvarmech}
\frac{d}{ds}\Big|_{0}\int_JL(c_s,c'_s)\,dt
=-\int_J G_{c(t)}\Bigl(\nabla_{\partial_t}c'+\vgrad V(c),\,A(t)\Bigr)dt ,
\end{equation}
and $c$ is critical for $\mathcal{L}$ if and only if
\begin{equation}\label{eq:mechEL}
\nabla_{\partial_t}c'=-\vgrad V(c),
\qquad\text{that is,}\qquad
c''=\Gamma_{c}(c',c')-\vgrad V(c).
\end{equation}
\end{theorem}

\begin{proof}
Write $w=\partial_t c_s$. Using \eqref{eq:metricityalongcurves} in the
$s-$direction, and then \eqref{eq:symmetrylemma},
\[
\partial_s\bigl[\tfrac12G_{c_s}(w,w)\bigr]
=G_{c_s}\bigl(\nabla_{\partial_s}w,w\bigr)
=G_{c_s}\bigl(\nabla_{\partial_t}\partial_sc_s,\,w\bigr),
\]
while $\partial_sV(c_s)=dV(\partial_sc_s)=G(\vgrad V(c_s),\partial_sc_s)$.
At $s=0$, with $A=\partial_s|_0c_s$, and using
\eqref{eq:metricityalongcurves} in the $t-$direction,
\begin{align*}
\partial_s\big|_0 L(c_s,c'_s)
=& G_c(\nabla_{\partial_t}A,c')-G_c(\vgrad V(c),A) \\
=& \partial_t\bigl[G_c(A,c')\bigr]
-G_c\bigl(A,\nabla_{\partial_t}c'+\vgrad V(c)\bigr).
\end{align*}
Integrating over $J$ and discarding the exact term (as $A$ vanishes near
$\partial J$) gives \eqref{eq:firstvarmech}, the interchange of
$\frac{d}{ds}\big|_0$ and $\int_J$ is justified as in the proof of
Theorem~\ref{thm:EL}.

If \eqref{eq:mechEL} holds, then \eqref{eq:firstvarmech} vanishes for all $A$
and $c$ is critical. Conversely, suppose that $c$ is critical and put
$w(t):=\nabla_{\partial_t}c'+\vgrad V(c)\in C^{\infty}(I,E)$. For each
$y\in E$ and each $\varphi\in C^{\infty}_c(I)$, the section
$A(t)=\varphi(t)\,y$ has compact support, so by
Lemma~\ref{lem:realisation} and \eqref{eq:firstvarmech},
$\int_I\varphi(t)\,G_{c(t)}(w(t),y)\,dt=0$. As $\varphi$ is arbitrary,
$G_{c(t)}(w(t),y)=0$ for all $t$ by continuity and, as $y\in E$ is
arbitrary too, weak non degeneracy gives $w(t)=0$ for every $t\in I$.
\end{proof}

\begin{remark}
For $(\Nn,G,\Gamma)=(\Mm,G^{c},\Gamma^{c})$, Theorem~\ref{thm:mechEL} with
$V=0$ recovers the geodesic equation $g''=\Gamma^{c}_g(g',g')$ of
\cite{GMi91}, that is, \eqref{eq:spray}. For $E=\R^{m}$, with $G$ a
(pseudo-)Riemannian metric and $V$ a potential, \eqref{eq:mechEL} is
Newton's equation, and Theorem~\ref{thm:EL} reduces to the classical
Euler--Lagrange equations. 
\end{remark}

\subsection{The energy function}\label{ssec:energy}

\begin{definition}\label{def:energy}
The \emph{energy} of $\mathcal{L}\in C^{\infty}(T\Nn,\R)$ is
$\Ee_{\mathcal{L}}\in C^{\infty}(T\Nn,\R)$, given by
\[
\Ee_{\mathcal{L}}(v):=F\mathcal{L}(v)(v)-\mathcal{L}(v)
=D_2L(\pi(v),v)(v)-L(\pi(v),v).
\]
For a mechanical Lagrangian, $F\mathcal{L}(g,k)(h)=G_g(k,h)$ and
the energy becomes $\Ee_{\mathcal{L}}(g,k)=\tfrac12G_g(k,k)+V(g)$.
\end{definition}

The following result expresses the conservation of energy along critical curves.

\begin{theorem}\label{thm:energy}
Along every critical curve $c$ of $\mathcal{L}$, the function given by
$t\mapsto\Ee_{\mathcal{L}}(c(t),c'(t))$ is constant.
\end{theorem}

\begin{proof}
Keep $f(t)=D_1L(c,c')$, $\gamma(t)=D_2L(c,c')$ as in
Theorem~\ref{thm:EL}, which asserts precisely condition (ii) of
Proposition~\ref{prop:DBR} for the pair $(f,\gamma)$. By
Lemma~\ref{lem:pairing} with $w=c'$ (which is $C^{1}$, indeed smooth),
$t\mapsto\gamma(t)(c'(t))$ is $C^{1}$ with derivative
$f(t)(c')+\gamma(t)(c'')$. On the other hand, by the chain rule
$\frac{d}{dt}L(c,c')=f(t)(c')+\gamma(t)(c'')$ as well. Subtracting,
$\frac{d}{dt}\Ee_{\mathcal{L}}(c,c')=0$.
\end{proof}

\begin{remark}
For time dependent Lagrangians $\mathcal{L}\in C^{\infty}(\R\times T\Nn)$
the same proofs give \eqref{eq:EL}, together with
$\frac{d}{dt}\Ee_{\mathcal{L}}(t,c,c')
=-\partial_t L(t,c(t),c'(t))$, although we will not need this.
\end{remark}

\section{Symmetry and conservation laws}
\label{sec:noether}

\subsection{The infinitesimal Noether theorem}\label{ssec:noetherthm}
On an arbitrary Fr\'echet manifold a smooth vector field need not admit a flow (Picard's
theorem fails, see \cite{Ham82,KM97}), so a Noether theorem whose hypothesis
is the invariance under a one-parameter group is unnecessarily
restrictive and, in general, empty. The right hypothesis is the infinitesimal
one. We will also allow invariance up to a boundary term, which gives the energy
theorem and the Galilean type symmetries as corollaries.

\begin{definition}\label{def:infinvariance}
Let $X\in C^{\infty}(\Nn,E)$ (a vector field on $\Nn$ in the global
chart), and let $F\in C^{\infty}(\Nn)$. We say that $\mathcal{L}$ is
\emph{$X-$invariant up to $F$} if, for all $(g,k)\in\Nn\times E$, it holds 
\begin{equation}\label{eq:quasiinv}
D_1L(g,k)\bigl(X(g)\bigr)+D_2L(g,k)\bigl(DX(g)(k)\bigr)=dF(g)(k)\,.
\end{equation}
When $F=0$ we say that $\mathcal{L}$ is \emph{(infinitesimally)
$X-$invariant}. The \emph{Noether function} of $X$ is the $J_X\in C^{\infty}(T\Nn)$
given by
\[
J_X(v):=F\mathcal{L}(v)\bigl(X_{\pi(v)}\bigr)=D_2L(\pi(v),v)\bigl(X(\pi(v))\bigr)\,.
\]
\end{definition}

Note that the left hand side of \eqref{eq:quasiinv} is the derivative of
$\mathcal{L}$ along the complete (tangent) lift of $X$. Of course, if $X$ does
have a flow $\psi_s$, then \eqref{eq:quasiinv} with $F=0$ is implied by
$\mathcal{L}\circ T\psi_s=\mathcal{L}$.

\begin{theorem}[Noether]\label{thm:noether}
If $\mathcal{L}$ is $X-$invariant up to $F$, then along every critical
curve $c$ of $\mathcal{L}$ the function
\[
t\longmapsto J_X\bigl(c'(t)\bigr)-F\bigl(c(t)\bigr)
\]
is constant. In particular, if $\mathcal{L}$ is $X-$invariant, then $J_X$ is a
constant of the motion.
\end{theorem}

\begin{proof}
With $f,\gamma$ as before and $w(t):=X(c(t))$ (a smooth curve in $E$ with
$w'(t)=DX(c(t))(c'(t))$ by the chain rule), Lemma~\ref{lem:pairing} gives
\begin{align*}
\frac{d}{dt}\,J_X(c'(t))
=& \frac{d}{dt}\bigl[\gamma(t)(w(t))\bigr]
=f(t)\bigl(X(c)\bigr)+\gamma(t)\bigl(DX(c)(c')\bigr)\\
=& dF(c)(c')
=\frac{d}{dt}F(c(t)),
\end{align*}
where the third equality is \eqref{eq:quasiinv} evaluated at
$(g,k)=(c(t),c'(t))$.
\end{proof}

\begin{remark}\label{rem:noetherremarks}
(a) $J_X(v)=F\mathcal{L}(v)(X_{\pi(v)})$ is chart independent, as a fibre
derivative (Remark~\ref{rem:ELremarks}(b)). (b) For a mechanical
Lagrangian, $J_X(g,k)=G_g(k,X(g))$. (c) Theorem~\ref{thm:energy} is the
degenerate case of the same computation, with the ``symmetry''
acting on time; we have kept it as a separate result for the sake of clarity.
\end{remark}

\subsection{The momentum map of the diffeomorphism group}
\label{ssec:momentum}
Consider the right action of $\Diff(M)$ on $\Mm$ by pullback,
$g\cdot\varphi=\varphi^{*}g$. The fundamental vector field of
$X\in\mathfrak{X}(M)$ is the map $\zeta_X\colon\Mm\to \Sym(M)$
given by
\[
\zeta_X(g)=\frac{d}{ds}\Big|_{0}\varphi_s^{*}g=\LD{X}g
=2\,\delta^{*}_g(X^{\flat}),
\]
where $\varphi_s$ is the flow of $X$ (which exists, $M$ being compact).
Notice that $\zeta_X$ is a smooth vector field on $\Mm$. In fact, in the global
chart the map $g\mapsto\LD{X}g$ is even linear in $g$, with
$D\zeta_X(g)(h)=\LD{X}h$.

\begin{proposition}\label{prop:momentum}
Let $\mathcal{L}(g,k)=\tfrac12G^{c}_g(k,k)-V(g)$ be a mechanical
Lagrangian on $(\Mm,G^{c})$ whose potential $V$ is a Riemannian
functional. Then:
\begin{enumerate}[label=\textit{\alph*)}]
\item\label{noe1} $\mathcal{L}$ is $\zeta_X$-invariant, in the sense of
Definition~\textup{\ref{def:infinvariance}}, for every
$X\in\mathfrak{X}(M)$;
\item\label{noe2} the associated Noether functions assemble into the map
$\mathbf{J}\colon T\Mm\to\mathfrak{X}(M)'$ given by
\[
\langle\mathbf{J}(g,k),X\rangle:=J_{\zeta_X}(g,k)
=G^{c}_g\bigl(k,\LD{X}g\bigr)
=2\int_M \ip{X^{\flat}}{\delta^{c}_gk}{g}\,\dv_g,
\]
where $\mathfrak{X}(M)'$ denotes the continuous dual of $\mathfrak{X}(M)$
(realized, for each $(g,k)$, by the smooth one-form density
$2\,(\delta^{c}_gk)\otimes\dv_g$ paired with vector fields), and
$\langle\mathbf{J}(g_t,g'_t),X\rangle$ is constant in $t$ along every
critical curve, for every $X$;
\item\label{noe3} for a fixed tangent vector $(g,k)$, we have
$\mathbf{J}(g,k)=0$ if and only if $\delta^{c}_gk=0$, if and only if $k$
is $G^{c}_g$-orthogonal to the image of the infinitesimal $\Diff(M)-$action
at $g$.
\end{enumerate}
\end{proposition}

\begin{proof}\mbox{}\\
\ref{noe1} 
Let $X\in\mathfrak X(M)$ and let $\varphi_s$ be its flow. The pullback action on $T\mathcal M$ is $(g,k)\mapsto(\varphi_s^*g,\varphi_s^*k)$. Its infinitesimal generator is \[ 
\left.\frac{d}{ds}\right|_{0} (\varphi_s^*g,\varphi_s^*k) = (\mathcal L_Xg,\mathcal L_Xk). 
\] 
In particular, $\zeta_X(g)=\mathcal L_Xg$, and $D\zeta_X(g)(k)=\mathcal L_Xk$, 
the latter identity following from the linearity of $g\mapsto\mathcal L_Xg$. 
The naturality of the pointwise product $\langle\cdot,\cdot\rangle_g^c$ and 
of the Riemannian density gives 
$G_{\varphi_s^*g}^c(\varphi_s^*k,\varphi_s^*k) = G_g^c(k,k)$.
Differentiating at $s=0$ and using the chain rule yields 
\[ 
D_1\bigl(G^c(k,k)\bigr)(g)(\mathcal L_Xg) + 2G_g^c(k,\mathcal L_Xk) = 0\,.
\] 
Equivalently, 
\[ 
\frac12D_1\bigl(G^c(k,k)\bigr)(g)(\mathcal L_Xg) + G_g^c(k,\mathcal L_Xk) = 0\,. 
\] 
Likewise, since $V$ is a Riemannian functional, $V(\varphi_s^*g)=V(g)$, and differentiation at $s=0$ gives $dV(g)(\mathcal L_Xg)=0$. For $L(g,k)=\frac12G_g^c(k,k)-V(g)$, we have 
\[ 
D_1L(g,k)(l) = \frac12D_1\bigl(G^c(k,k)\bigr)(g)(l)-dV(g)(l) 
\] 
and \[ D_2L(g,k)(m)=G_g^c(k,m). \] Consequently, \[ \begin{aligned} &D_1L(g,k)\bigl(\zeta_X(g)\bigr) +D_2L(g,k)\bigl(D\zeta_X(g)(k)\bigr) \\ &= \frac12D_1\bigl(G^c(k,k)\bigr)(g)(\mathcal L_Xg) -dV(g)(\mathcal L_Xg) +G_g^c(k,\mathcal L_Xk) =0. \end{aligned} \] This is precisely the infinitesimal invariance condition \eqref{eq:quasiinv} with $F=0$ and with $\zeta_X$ as the symmetry vector field.

\noindent\ref{noe2} For a mechanical Lagrangian, 
\[ 
F\mathcal L(g,k)(h)=D_2L(g,k)(h)=G_g^c(k,h). 
\] 
Hence 
\[ 
J_{\zeta_X}(g,k) = F\mathcal L(g,k)\bigl(\zeta_X(g)\bigr) = G_g^c(k,\mathcal L_Xg). 
\] 
Since $\mathcal L_Xg=2\delta_g^*(X^\flat)$, the definition of the formal adjoint $\delta_g^c$ gives 
\[ 
\begin{aligned} 
J_{\zeta_X}(g,k) &= 2G_g^c\bigl(k,\delta_g^*(X^\flat)\bigr) 
= 2\int_M \left\langle X^\flat,\delta_g^c k\right\rangle_g\,\mu_g 
= 2\int_M g\bigl(X,(\delta_g^c k)^\sharp\bigr)\,\mu_g. 
\end{aligned} 
\] 
This defines $\mathbf J(g,k)\in\mathfrak X(M)'$ by 
$\left\langle\mathbf J(g,k),X\right\rangle := J_{\zeta_X}(g,k)$.
By part~\ref{noe1}, $\mathcal L$ is $\zeta_X$-invariant for every fixed $X\in\mathfrak X(M)$. Theorem~\ref{thm:noether} therefore shows that 
$t\mapsto \left\langle\mathbf J(g_t,g_t'),X\right\rangle$ is constant along every critical curve, for every fixed $X$.

\noindent\ref{noe3} If $\delta_g^c k=0$, the preceding integral formula gives $\mathbf J(g,k)=0$. Conversely, if $\mathbf J(g,k)=0$, then \[ \int_M g\bigl(X,(\delta_g^c k)^\sharp\bigr)\,\mu_g=0 \] for every $X\in\mathfrak X(M)$. Taking $X=(\delta_g^c k)^\sharp$ yields \[ 0 = \int_M \lvert\delta_g^c k\rvert_g^2\,\mu_g. \] Since the integrand is nonnegative, it follows that $\delta_g^c k=0$ identically. Finally, the image of the infinitesimal action at $g$ is \[ \{\mathcal L_Xg:X\in\mathfrak X(M)\} = \operatorname{im}(2\delta_g^*). \] Thus $k$ is $G_g^c$-orthogonal to this image if and only if \[ 0 = G_g^c(k,\mathcal L_Xg) = 2\int_M \left\langle X^\flat,\delta_g^c k\right\rangle_g\,\mu_g \] for every $X$, which is equivalent to $\delta_g^c k=0$. Therefore $\mathbf J(g,k)=0$  if and only if  $\delta_g^c k=0$  if and only if   $k\perp_{G_g^c}\operatorname{im}(2\delta_g^*)$.
\end{proof}

\begin{proposition}[Equivariance]\label{prop:equivariance}
For all $\varphi\in\Diff(M)$, $X\in\mathfrak{X}(M)$ and
$(g,k)\in T\Mm$,
\[
\bigl\langle\mathbf{J}(\varphi^{*}g,\varphi^{*}k),\,X\bigr\rangle
=\bigl\langle\mathbf{J}(g,k),\,\varphi_{*}X\bigr\rangle .
\]
Thus $\mathbf{J}$ is equivariant for the pullback action on $T\Mm$ and
the (co)adjoint type action $X\mapsto\varphi_*X$ on $\mathfrak{X}(M)$.
\end{proposition}

\begin{proof}
$\LD{X}(\varphi^{*}g)=\varphi^{*}\bigl(\LD{\varphi_{*}X}g\bigr)$ and the
$\Diff$-invariance of $G^{c}$ give the identity
$G^{c}_{\varphi^{*}g}(\varphi^{*}k,\LD{X}\varphi^{*}g)
=G^{c}_{g}(k,\LD{\varphi_*X}g)$.
\end{proof}

\begin{remark}\label{rem:FMcomparison}
Proposition~\ref{prop:momentum} is the Lagrangian counterpart of
\cite[\S6]{FM72b}, where the conserved quantity
$(g,k)\mapsto\mathcal{G}_g(k,-\LD{Z}g)$ is derived from the Hamiltonian
conservation scheme. We use the term ``Noether momentum map'' in the sense made
precise above, that is, $\mathbf{J}$ is the assembly of the
$\Diff(M)-$Noether charges, and
Propositions~\ref{prop:momentum}--\ref{prop:equivariance} establish exactly
its charge pairing, its conservation and its equivariance. However, we do not 
consider here the canonical (pre)symplectic form on $T\Mm$, nor do we verify the moment
map identity $\iota_{\zeta_X}\omega=d\langle\mathbf{J},X\rangle$. In the Hamiltonian
picture of \cite{FM72b,MR99}, $\mathbf{J}$ is expected to coincide with the
momentum map of the cotangent lifted action, and the fact that the constraint
surface of geometrodynamics is one of its zero level sets underlies the
reduction picture of superspace \cite{Giu09}.
\end{remark}

\section{The Einstein equations as a Lagrangian dynamical system}
\label{sec:einstein}

\subsection{$(n+1)-$geometry and curves of metrics}
\label{ssec:setup}
Let $\widetilde M=I\times M$ and let $g_t$, $t\in I$, be a smooth curve in
$\Mm$. Fix a real constant $\alpha\neq0$ and define the pseudo-Riemannian
metric $\widetilde g=\widetilde g(\alpha;g_\bullet)$ on $\widetilde M$ by
\begin{equation}\label{eq:tildeg}
\widetilde g\big|_{T_pM\times T_pM}=g_t,\qquad
\widetilde g(\partial_t,\cdot)\big|_{TM}=0,\qquad
\widetilde g(\partial_t,\partial_t)=\alpha ;
\end{equation}
in adapted coordinates,
$\widetilde g=\alpha\,dt^2+(g_t)_{ij}\,dx^i dx^j$. The case $\alpha=-N^2<0$
is the Lorentzian metric of the $3{+}1$ (ADM) formalism with constant
lapse $N$ and vanishing shift \cite{ADM62,FM72b}, while $\alpha>0$ produces a
Riemannian ``spacetime''. This is the situation studied by Gil-Medrano
\cite{GM96}, whose main results we now quote in the notation of
\S\ref{ssec:christoffel}. Recall $\tau_\lambda=\tau-\lambda(n-1)$ and
$S_\lambda$ from \eqref{eq:Slambda}, and define the
\emph{energy density} (or Hamiltonian density)
\begin{equation}\label{eq:Hamdensity}
\Ham_\lambda(g,k):=\tfrac12\,\ip{k}{k}{g}^{-}+2\alpha\,\tau_\lambda(g)
\in C^{\infty}(M),\qquad (g,k)\in T\Mm .
\end{equation}

\begin{theorem}[{Gil-Medrano \cite[Props.~4.3 and 4.5]{GM96}}]
\label{thm:GMtheorem}
Let $n\ge2$, $\lambda\in\R$, $\alpha\neq0$, and let $g_t$, $t\in I$, be a
curve in $\Mm$. The metric $\widetilde g$ of \eqref{eq:tildeg} is Einstein
with constant $\lambda$, that is, $\rho(\widetilde g)=\lambda\widetilde g$,
if and only if the curve verifies the \emph{evolution equation}
\begin{equation}\label{eq:evolution}
g''=\Gamma^{-}_{g}(g',g')-2\alpha\,\vgrad^{-}S_\lambda(g)
\end{equation}
together with the \emph{constraints}
\begin{equation}\label{eq:constraints}
\delta^{-}_{g}g'=0
\qquad\text{and}\qquad
\Ham_\lambda(g,g')=0 .
\end{equation}
For $\lambda=0$, $\alpha=-1$, $n=3$ this is the vacuum Einstein system of
\cite{DeW67,FM72b}. Moreover \cite[Prop.~4.4]{GM96}, within the family
$G^{c}$ the DeWitt value $c=c^{-}$ is the only one for which an
equivalence of this form can hold, because for $c\neq c^{-}$ the analogous
system forces $g_t\equiv g_0$ with $g_0$ Einstein.
\end{theorem}

Explicitly, by \eqref{eq:sprayminus} and \eqref{eq:gradSlambda},
\eqref{eq:evolution} reads
\begin{equation}\label{eq:evolutionexplicit}
g''=g'g^{-1}g'-\tfrac12\tr_g(g')\,g'
-\tfrac{1}{4(n-1)}\ip{g'}{g'}{g}^{-}\,g
+2\alpha\,\rho(g)-\tfrac{\alpha}{n-1}\,\tau(g)\,g-\alpha\lambda\,g .
\end{equation}

\subsection{The constraints as Noether charges}\label{ssec:charges}
Theorem~\ref{thm:mechEL} allows us to convert \eqref{eq:evolution} into a
variational statement.

\begin{proposition}\label{prop:chargeident}
Define the \emph{DeWitt Lagrangian} $\mathcal{L}_{\alpha,\lambda}\in C^{\infty}(T\Mm)$
as
\begin{equation}\label{eq:deWittLag}
\mathcal{L}_{\alpha,\lambda}(g,k)
=\tfrac12\,G^{-}_g(k,k)-2\alpha\,S_\lambda(g)\,,
\end{equation}
that is, a mechanical-type Lagrangian on $(\Mm,G^{-},\Gamma^{-})$ with potential
$V=2\alpha S_\lambda$. Then:
\begin{enumerate}[label=\textit{\alph*)}]
\item\label{var1} a curve $g_t$ is critical for $\mathcal{L}_{\alpha,\lambda}$ if and
only if it verifies the evolution equation \eqref{eq:evolution};
\item\label{var2} the energy of $\mathcal{L}_{\alpha,\lambda}$ is the integrated
energy density,
\[
\Ee_{\mathcal{L}_{\alpha,\lambda}}(g,k)
=\tfrac12G^{-}_g(k,k)+2\alpha S_\lambda(g)
=\int_M\Ham_\lambda(g,k)\,\dv_g ;
\]
\item\label{var3} the momentum map of Proposition~\textup{\ref{prop:momentum}} 
(for $c=c^{-}$) is
\[
\langle\mathbf{J}(g,k),X\rangle
=2\int_M\ip{X^{\flat}}{\delta^{-}_gk}{g}\dv_g\,,
\]
and the constraints \eqref{eq:constraints} say precisely that
\[
\mathbf{J}(g,g')=0\ 
\qquad\text{and}\qquad
\Ham_\lambda(g,g')=0 \textup{ pointwise on } M
\]
(the \emph{momentum constraint} and the \emph{Hamiltonian constraint}, respectively.) 
In particular the Hamiltonian constraint is a pointwise refinement
of the vanishing of the Noether energy.
\end{enumerate}
\end{proposition}

\begin{proof}\mbox{}\\
\ref{var1} $V=2\alpha S_\lambda$ admits the $G^{-}-$gradient
$2\alpha\,\vgrad^{-}S_\lambda$ by \eqref{eq:gradSlambda}, so
Theorem~\ref{thm:mechEL} applies and \eqref{eq:mechEL} is
\eqref{eq:evolution}. 

\noindent\ref{var2} is Definition~\ref{def:energy} together with
\eqref{eq:Hamdensity} and $G^{-}_g(k,k)=\int\ip{k}{k}{g}^{-}\dv_g$. 

\noindent\ref{var3} is Proposition~\ref{prop:momentum}\,\ref{noe2}--\ref{noe3}, $S_\lambda$ being a
Riemannian functional.
\end{proof}

\begin{remark}\label{rem:physics}
The physical interpretation of this result is the classical one 
\cite{DeW67,FM72b,Giu09}.
Geometrodynamics is the motion of a ``point'' $g$ in the superspace
potential $-2\alpha S_\lambda$, and the general covariance manifests itself in
the requirement that the motion have vanishing $\Diff(M)-$momentum and
vanishing energy \emph{density}, and not merely vanishing total energy.
Misner observed, in the context of a topologically invariant quantum theory of gravity, 
that the generator of infinitesimal deformations of a hypersurface must vanish; see 
\cite[Sec.~7, especially p.~508]{Mis57}. Fischer and Marsden gave a classical formulation 
of the corresponding local statement in \cite[Thm.~7.1]{FM72b}, where covariance under 
local changes of the time parametrization leads to the pointwise Hamiltonian constraint. 
In our constant-lapse, zero-shift setting, this condition is imposed on the initial data, 
and we prove next that it is propagated by the evolution equation.
\end{remark}

\subsection{Propagation of the momentum constraint}\label{ssec:momprop}

\begin{theorem}\label{thm:momprop}
Let $g_t$, $t\in I$, be any critical curve of
$\mathcal{L}_{\alpha,\lambda}$ (equivalently, any solution of
\eqref{eq:evolution}). If $\delta^{-}_{g_{t_0}}g'_{t_0}=0$ for some
$t_0\in I$, then
\[
\delta^{-}_{g_t}g'_t=0\qquad\text{for all }t\in I .
\]
Equivalently, every smooth solution of \eqref{eq:evolution} preserves the
momentum constraint set $\{\mathbf{J}=0\}$, so a solution which starts
$G^{-}-$orthogonal to the orbits of the $\Diff(M)$ action, stays so forever.
\end{theorem}

\begin{proof}
By Propositions~\ref{prop:momentum} and \ref{prop:chargeident}\ref{var3}, for
each fixed $X\in\mathfrak{X}(M)$ the function
$t\mapsto\langle\mathbf{J}(g_t,g'_t),X\rangle
=2\int_M\ip{X^{\flat}}{\delta^{-}_{g_t}g'_t}{g_t}\dv_{g_t}$ is constant
along the critical curve. At $t=t_0$ it vanishes for every $X$, hence it
vanishes identically for every $X$ and every $t$, and
Proposition~\ref{prop:momentum}\ref{noe3} gives $\delta^{-}_{g_t}g'_t=0$ for
each $t$.
\end{proof}

\begin{remark}
This argument replaces the reasoning with linear systems of partial
differential equations of \cite[Prop.~3.5]{FM72b} (for this half of the
constraints) by symmetry considerations, and it does not need any
uniqueness theory at all; in a sense, we could say that the Noether
constants do all the work. We think this is more convenient from a
geometric point of view.
\end{remark}

\subsection{A pointwise transport law for the energy density}
\label{ssec:transport}
The next theorem and its corollaries are the main technical result of this section. 
The theorem holds along every solution of the evolution equation, 
without assuming any constraint, and for all $n\ge2$, $\alpha\neq0$, $\lambda\in\R$.

\begin{theorem}[Transport identity]\label{thm:transport}
Along every solution $g_t$ of the evolution equation
\eqref{eq:evolution}, with $k:=g'$, the energy density
\eqref{eq:Hamdensity} verifies
\begin{equation}\label{eq:transport}
\partial_t\,\Ham_\lambda(g,k)
=-\tfrac12\,\tr_g(k)\,\Ham_\lambda(g,k)
+2\alpha\,\delta_g\bigl(\delta^{-}_{g}k\bigr),
\end{equation}
equivalently, since $\partial_t\dv_{g}=\tfrac12\tr_g(k)\,\dv_g$,
\begin{equation}\label{eq:transportdensity}
\partial_t\Bigl(\Ham_\lambda(g,k)\,\dv_{g}\Bigr)
=2\alpha\,\delta_g\bigl(\delta^{-}_{g}k\bigr)\,\dv_{g} .
\end{equation}
\end{theorem}

\begin{proof}
Write $P:=g^{-1}k$ (a field of endomorphisms), so
$\tr P=\tr_gk$, $\tr P^2=\ip{k}{k}{g}$ and, by
Lemma~\ref{lem:cidentities}\ref{aux1},
$\ip{k}{k}{g}^{-}=\tr P^{2}-(\tr P)^{2}$. From $\partial_tg=k$ we get
$\partial_t(g^{-1})=-g^{-1}kg^{-1}$, hence
\begin{equation}\label{eq:Pdot}
\partial_tP=-P^{2}+g^{-1}\dot k,
\end{equation}
where the dot denotes derivative with respect to $t$ and, 
by \eqref{eq:evolutionexplicit},
\begin{equation}\label{eq:kdot}
g^{-1}\dot k
=P^{2}-\tfrac12(\tr P)\,P
-\tfrac{1}{4(n-1)}\ip{k}{k}{g}^{-}\,\mathrm{Id}
+2\alpha\,g^{-1}\rho-\tfrac{\alpha}{n-1}\,\tau\,\mathrm{Id}
-\alpha\lambda\,\mathrm{Id}.
\end{equation}

Let us consider the two terms in \eqref{eq:Hamdensity} separately; 
first, the kinetic term
$\partial_t\ip{k}{k}{g}^{-}$. Using \eqref{eq:Pdot},
$\partial_t\tr P^{2}=2\tr(P\,\partial_tP)=-2\tr P^{3}+2\tr(Pg^{-1}\dot k)$
and $\partial_t(\tr P)^{2}=2(\tr P)\bigl(-\tr P^{2}+\tr(g^{-1}\dot k)\bigr)$.
Taking traces of \eqref{eq:kdot} (against $P$ and plainly, respectively), with
$\tr(Pg^{-1}\rho)=\ip{\rho}{k}{g}$ and $\tr(g^{-1}\rho)=\tau$, we get
\begin{align*}
2\tr\bigl(Pg^{-1}\dot k\bigr)
=& 2\tr P^{3}-(\tr P)\tr P^{2}
-\tfrac{1}{2(n-1)}\ip{k}{k}{g}^{-}\tr P\\
& +4\alpha\ip{\rho}{k}{g}
-\tfrac{2\alpha}{n-1}\,\tau\,\tr P-2\alpha\lambda\,\tr P,
\end{align*}
and
\begin{align*}
\tr\bigl(g^{-1}\dot k\bigr)
&=\tr P^{2}-\tfrac12(\tr P)^{2}
-\tfrac{n}{4(n-1)}\ip{k}{k}{g}^{-}
+\tfrac{n-2}{n-1}\,\alpha\tau-n\alpha\lambda ,
\end{align*}
where we have used $2\alpha\tau-\frac{n\alpha}{n-1}\tau
=\frac{n-2}{n-1}\alpha\tau$. Therefore
\begin{align*}
\partial_t\ip{k}{k}{g}^{-}
&=\partial_t\tr P^{2}-\partial_t(\tr P)^{2} \\[2pt]
&=\Bigl[
   -(\tr P)\tr P^{2}
   -\frac{1}{2(n-1)}
      \ip{k}{k}{g}^{-}\tr P
   +4\alpha\ip{\rho}{k}{g}
   -\frac{2\alpha}{n-1}\tau\tr P
   -2\alpha\lambda\tr P
 \Bigr]
\\
&\quad+\Bigl[
   (\tr P)^{3}
   +\frac{n}{2(n-1)}
      \ip{k}{k}{g}^{-}\tr P
   -\frac{2(n-2)}{n-1}\alpha\tau\tr P
   +2n\alpha\lambda\tr P
 \Bigr]
\\
&=-(\tr P)\bigl[\tr P^{2}-(\tr P)^{2}\bigr]
  +\frac{n-1}{2(n-1)}
     \ip{k}{k}{g}^{-}\tr P
  +4\alpha\ip{\rho}{k}{g}
\\[-2pt]
&\qquad
  -2\alpha
     \Bigl[\frac{1+(n-2)}{n-1}\Bigr]\tau\tr P
  +2(n-1)\alpha\lambda\tr P .
\end{align*}
that is (taking into account that $\tr P^{2}-(\tr P)^{2}=\ip{k}{k}{g}^{-}$ and
after simplifying numerical factors),
\begin{equation}\label{eq:kineticdot}
\partial_t\ip{k}{k}{g}^{-}
=-\tfrac12(\tr_gk)\ip{k}{k}{g}^{-}
+4\alpha\ip{\rho(g)}{k}{g}
-2\alpha\,\tau\,\tr_gk
+2(n-1)\alpha\lambda\,\tr_gk .
\end{equation}

Now we consider the potential term in \eqref{eq:Hamdensity}. By \eqref{eq:variations},
\begin{equation}\label{eq:taudot}
\partial_t\tau_\lambda(g)=\partial_t\tau(g)=D\tau(g)(k)
=\Delta_g(\tr_gk)+\delta_g\delta_gk-\ip{\rho(g)}{k}{g},
\end{equation}
where $\delta_g\delta_gk$ denotes $\delta_g(\delta_gk)$, the divergence of
the one-form $\delta_gk$.

We can assemble the previous computations. Multipliying $1/2$ by
\eqref{eq:kineticdot}, and adding the result to twice \eqref{eq:taudot}, we get:
\[
\partial_t\Ham_\lambda
=-\tfrac14(\tr_gk)\ip{k}{k}{g}^{-}
-\alpha\bigl(\tau-(n-1)\lambda\bigr)\tr_gk
+2\alpha\bigl[\Delta_g(\tr_gk)+\delta_g\delta_gk\bigr],
\]
the curvature terms $\pm2\alpha\ip{\rho}{k}{g}$ having cancelled. Since
$\tau-(n-1)\lambda=\tau_\lambda$, the first two terms equal
$-\tfrac12(\tr_gk)\bigl[\tfrac12\ip{k}{k}{g}^{-}+2\alpha\tau_\lambda\bigr]
=-\tfrac12(\tr_gk)\,\Ham_\lambda$. Finally, by \eqref{eq:deltaminus} and
$\Delta_g=\delta_gd$ on functions,
\[
\Delta_g(\tr_gk)+\delta_g\delta_gk
=\delta_g\bigl(d(\tr_gk)+\delta_gk\bigr)
=\delta_g\bigl(\delta^{-}_{g}k\bigr),
\]
which proves \eqref{eq:transport}. As for \eqref{eq:transportdensity},
$\partial_t\dv_{g}=\tfrac12\tr_g(k)\dv_g$ by \eqref{eq:variations}, so
$\partial_t(\Ham_\lambda\dv_g)
=(\partial_t\Ham_\lambda)\dv_g+\tfrac12\tr_g(k)\Ham_\lambda\dv_g
=2\alpha\,\delta_g(\delta^{-}_gk)\,\dv_g$.
\end{proof}

\begin{corollary}[Pointwise conservation law]\label{cor:pointwise}
Along every solution of \eqref{eq:evolution} verifying the momentum
constraint $\delta^{-}_{g_t}g'_t=0$ (for instance, by
Theorem~\textup{\ref{thm:momprop}}, any solution verifying it at one
instant), the density
\[
\Ham_\lambda(g_t,g'_t)\,\dv_{g_t}
\]
does not depend on $t$, pointwise on $M$. Equivalently,
\[
\Ham_\lambda(g_t,g'_t)(p)
=\Ham_\lambda(g_{t_0},g'_{t_0})(p)\,
\exp\bigl(-\tfrac12\int_{t_0}^{t}\tr_{g_s}(g'_s)(p)\,ds\bigr)
\]
for every $p\in M$.
\end{corollary}

\begin{proof}
Immediate from \eqref{eq:transportdensity} and \eqref{eq:transport} with
$\delta^{-}_gk=0$. The second form is the solution of the linear scalar
ordinary differential equation
$\partial_t\Ham_\lambda=-\tfrac12\tr_g(k)\Ham_\lambda$ at each fixed
$p\in M$.
\end{proof}

\begin{corollary}[Propagation of the Hamiltonian constraint]
\label{cor:hamprop}
Let $g_t$ solve \eqref{eq:evolution} with
$\delta^{-}_{g_{t_0}}g'_{t_0}=0$ and $\Ham_\lambda(g_{t_0},g'_{t_0})=0$
for some $t_0\in I$. Then both constraints \eqref{eq:constraints} hold for
all $t\in I$, hence, by Theorem~\textup{\ref{thm:GMtheorem}}, the metric
$\widetilde g$ built from $g_t$ is Einstein with constant $\lambda$ on all
of $I\times M$.
\end{corollary}

\begin{proof}
Theorem~\ref{thm:momprop} propagates the momentum constraint, and
Corollary~\ref{cor:pointwise} then propagates the vanishing of
$\Ham_\lambda$.
\end{proof}

\begin{corollary}[Total energy]\label{cor:totalenergy}
Along every solution of \eqref{eq:evolution} the total energy
$\Ee(t)=\int_M\Ham_\lambda(g_t,g'_t)\dv_{g_t}$ is constant, with or
without the momentum constraint.
\end{corollary}

\begin{proof}
One may either quote Theorem~\ref{thm:energy} and
Proposition~\ref{prop:chargeident}\ref{var2}, or integrate
\eqref{eq:transportdensity} over the closed manifold $M$ and use
$\int_M\delta_g(\omega)\,\dv_g=0$.
\end{proof}

\begin{remark}\label{rem:transportremarks}
(a) In the case $\alpha=-1$, $\lambda=0$, and $n=3$, the scalar form of \eqref{eq:transport} is \[ \partial_t\Ham +\frac12\tr_g(k)\Ham +2\,\delta_g(\delta_g^-k)=0. \] Up to the translation between their momentum-density variable $\pi$ and our $\delta_g^-k$, this is the Lagrangian counterpart of the continuity equation of Fischer and Marsden \cite[Prop.~3.4, eq.~(3.10)]{FM72b}. Since \[ \partial_t\dv_g = \frac12\tr_g(k)\dv_g, \] the preceding scalar equation is equivalently written as the density identity \[ \partial_t(\Ham\dv_g) +2\,\delta_g(\delta_g^-k)\dv_g=0, \] which is precisely \eqref{eq:transportdensity} for $\alpha=-1$. Thus the term $\frac12\tr_g(k)\Ham$ in the scalar equation is absorbed by differentiating the evolving volume density. 
The relations between Noether momentum map and the momentum
constraint on the one hand, and between transport ODE and the Hamiltonian constraint
on the other, avoid the appeal to uniqueness for first order linear systems of partial 
differential equations which is implicit in \cite[Prop.~3.5]{FM72b}. (b)
Corollary~\ref{cor:pointwise} is strictly stronger than the conservation of
the Noether energy, because the conserved object is a \emph{density on $M$},
so there is one conservation law for each point. The naive mechanical analogy
only predicts the conservation of the integral. (c) Nothing in
Theorems~\ref{thm:momprop}--\ref{thm:transport} requires $\lambda=0$,
Lorentzian signature, or $n=3$. In particular, the ``Riemannian
geometrodynamics'' of \cite{GM96} has the same conservation laws,
as illustrated in Section~\ref{sec:examples}. (d) The existence of solutions
of \eqref{eq:evolution} for given initial data $(g_0,k_0)$ is not
addressed here, and the results above are identities for smooth
solutions. Let us note that the displayed equation is \emph{not} an ordinary
differential equation on a Banach manifold of $C^{k}-$metrics. Its right hand
side contains $\rho(g)$, which involves two spatial derivatives of $g$, so
$g\mapsto\rho(g)$ maps $C^{k}\to C^{k-2}$ and the Picard--Lindel\"of theory
does not apply; besides, the Einstein equations carry a diffeomorphism gauge
freedom. The local existence of Einstein developments for constrained initial
data is furnished instead by the standard gauge-fixed hyperbolic Cauchy
theory \cite{FM72a,CB09,Rin09}, to which we refer for the well-posedness
that our formal identities presuppose.
\end{remark}

\section{A Maupertuis--Jacobi theorem and the geodesic form of gravity}
\label{sec:jacobi}
Let us recall that the classical Maupertuis--Jacobi principle states that, at fixed energy $e$, the unparametrized trajectories of a mechanical system are geodesics of the conformally rescaled Jacobi metric $2(e-V)G$ on the region where the conformal factor does not vanish. DeWitt interpreted Einstein evolution as motion in the space of spatial geometries and expressed it as a geodesic equation for the DeWitt metric modified by a curvature-dependent force term \cite[\S 6, especially Eq.~(6.23)]{DeW67}. Here we show that, for the constant-lapse DeWitt mechanical system, this force term can instead be absorbed into a conformal rescaling of the metric and an explicit reparametrization. We first establish the corresponding Maupertuis--Jacobi theorem for weak pseudo-Riemannian metrics on Fr\'echet manifolds. The only subtlety
which is absent from the finite dimensional theory is that we must check first
that a \emph{weak} metric maintains that character after a conformal rescaling.

\subsection{Conformal rescaling of a weak metric}\label{ssec:conformal}

\begin{lemma}\label{lem:conformal}
Let $(G,\Gamma)$ be a weak pseudo-Riemannian structure with Levi-Civita
connection on $\Nn\subset E$ (Definition~\textup{\ref{def:mechanical}}),
and let $f\in C^{\infty}(\Nn)$ be nowhere zero and admit a $G-$gradient
$\vgrad f$. Then $\widetilde G:=fG$ is again weakly non degenerate and
admits the Levi-Civita connection
\begin{equation}\label{eq:conformalconn}
\widetilde\Gamma_g(h,k)
=\Gamma_g(h,k)
-\frac{1}{2f(g)}\Bigl[df(g)(h)\,k+df(g)(k)\,h-G_g(h,k)\,\vgrad f(g)\Bigr].
\end{equation}
\end{lemma}

\begin{proof}
The weak non degeneracy of $\widetilde G=fG$ is clear, as $f$ is nowhere
zero. The map $\widetilde\Gamma$ defined by \eqref{eq:conformalconn} is
smooth and symmetric so, by Lemma~\ref{lem:LCunique}, it suffices to verify
the metricity identity \eqref{eq:metricity} for $(\widetilde G,
\widetilde\Gamma)$. Since $D\widetilde G(g)(l)(h,k)
=df(g)(l)\,G_g(h,k)+f(g)\,DG(g)(l)(h,k)$ and, using
\eqref{eq:metricity} for $(G,\Gamma)$,
\[
f\,DG(l)(h,k)
=-f\,G(\Gamma(l,h),k)-f\,G(h,\Gamma(l,k)),
\]
what we must show is that
\[
df(l)\,G(h,k)
-f\,G(\Gamma(l,h),k)-f\,G(h,\Gamma(l,k))
+\widetilde G(\widetilde\Gamma(l,h),k)
+\widetilde G(h,\widetilde\Gamma(l,k))=0 .
\]
Now $\widetilde G(\widetilde\Gamma(l,h),k)=f\,G(\Gamma(l,h),k)
-\tfrac12[df(l)G(h,k)+df(h)G(l,k)-G(l,h)\,df(k)]$, using
$G(\vgrad f,k)=df(k)$, and symmetrically for the last term. Adding, the
$f\,G(\Gamma,\cdot)$ terms cancel and the bracket contributions sum to
\begin{align*}
&-\tfrac12\bigl[df(l)G(h,k)+df(h)G(l,k)-df(k)G(l,h)\bigr]\\
&-\tfrac12\bigl[df(l)G(h,k)+df(k)G(l,h)-df(h)G(l,k)\bigr]\\
&= -df(l)\,G(h,k),
\end{align*}
which cancels the leading $df(l)G(h,k)$. Hence \eqref{eq:metricity} holds
for $(\widetilde G,\widetilde\Gamma)$.
\end{proof}

\subsection{The Maupertuis--Jacobi theorem}\label{ssec:jacobithm}

\begin{theorem}
\label{thm:jacobi}
Let $\mathcal{L}=\tfrac12G(v,v)-V\circ\pi$ be a mechanical Lagrangian on
$\Nn$ (Definition~\textup{\ref{def:mechanical}}) with $V$ admitting a
$G-$gradient, let $e\in\R$, and put
\[
\Nn_e:=\{g\in\Nn: V(g)\neq e\},\qquad
\widetilde G:=2(e-V)\,G \ \text{ on }\ \Nn_e,
\]
a weak pseudo-Riemannian metric with Levi-Civita connection
$\widetilde\Gamma$ given by Lemma~\textup{\ref{lem:conformal}}
(with $f=2(e-V)$, $\vgrad f=-2\vgrad V$).
\begin{enumerate}[label=\textit{\alph*)}]
\item\label{jac1} Let $c\colon I\to\Nn_e$ be an extremal of $\mathcal{L}$ with 
constant energy $\Ee_{\mathcal{L}}(c,c')\equiv e$. Reparametrize by
$s(t)=\int_{t_0}^{t}2(e-V(c(r)))\,dr$, that is, $ds/dt=2(e-V(c))=f(c)$, and
write $\gamma(s)=c(t(s))$. Then $\gamma$ is a geodesic of $\widetilde G$,
of constant $\widetilde G-$speed $\widetilde G(\gamma',\gamma')\equiv1$
(where $'=d/ds$).
\item\label{jac2} Conversely, let $\gamma\colon J\to\Nn_e$ be a geodesic of
$\widetilde G$ with $\widetilde G(\gamma',\gamma')\equiv1$. Reparametrize
by $dt/ds=1/f(\gamma)=1/2(e-V(\gamma))$ and put $c(t)=\gamma(s(t))$. Then
$c$ is an extremal of $\mathcal{L}$ with energy $e$.
\end{enumerate}
\end{theorem}

\begin{proof}
Write $f:=2(e-V)$. It is nowhere zero on $\Nn_e=\{V\neq e\}$ by definition
and, along the connected curve under consideration (which lies in a
single connected component of $\Nn_e$), it has therefore a fixed sign, although
the sign may differ between components of $\Nn_e$. Also
$\vgrad f=-2\vgrad V$. Denote by $\nabla$ the $G-$Levi-Civita connection
and by $\widetilde\nabla$ the $\widetilde G-$one, whose Christoffel map is
\eqref{eq:conformalconn} with this $f$. In both parts we put
$\gamma(s)=c(t(s))$ with $ds/dt=f(c)$, so that $c'=f\,\gamma'$, that is,
$\gamma'=c'/f$, and $\frac{d}{ds}=\frac1f\frac{d}{dt}$.

We claim that, for \emph{any} smooth curve $c$
in $\Nn_e$ reparametrized as above, the following kinematical identity holds:
\begin{equation}\label{eq:jacobikey}
\widetilde\nabla_{\partial_s}\gamma'
=\frac{1}{f^2}\,\nabla_{\partial_t}c'
+\frac{G(c',c')}{f^3}\,\vgrad V(c) .
\end{equation}
Indeed, writing $\dot f:=\frac{d}{dt}f(c)=df(c)(c')$, the chain rule gives
$\gamma''=\frac{d}{ds}(c'/f)=\frac1f\frac{d}{dt}(c'/f)
=\frac{c''}{f^2}-\frac{\dot f}{f^3}c'$. On the other hand, by
\eqref{eq:conformalconn} with $h=k=\gamma'$, using the bilinearity
$\Gamma(\gamma',\gamma')=f^{-2}\Gamma(c',c')$, the relation
$df(\gamma')=f^{-1}\dot f$ (so that $f^{-1}df(\gamma')\,\gamma'
=\dot f\,f^{-3}c'$), and $\vgrad f=-2\vgrad V$ together with
$G(\gamma',\gamma')=f^{-2}G(c',c')$,
\[
\widetilde\Gamma_\gamma(\gamma',\gamma')
=\frac{1}{f^2}\Gamma(c',c')-\frac{\dot f}{f^3}c'
-\frac{G(c',c')}{f^3}\vgrad V(c) .
\]
Subtracting, the $\dot f\,f^{-3}c'$ terms cancel and
$\widetilde\nabla_{\partial_s}\gamma'
=\gamma''-\widetilde\Gamma_\gamma(\gamma',\gamma')
=\frac{1}{f^2}\bigl[c''-\Gamma(c',c')\bigr]+\frac{G(c',c')}{f^3}\vgrad V
$, which is \eqref{eq:jacobikey} since $c''-\Gamma(c',c')
=\nabla_{\partial_t}c'$.

\ref{jac1} Let $c$ be an extremal of energy $e$. By
Theorem~\ref{thm:mechEL}, $\nabla_{\partial_t}c'=-\vgrad V(c)$, and the
energy relation $\Ee_{\mathcal L}(c,c')=\tfrac12G(c',c')+V(c)=e$ reads
\begin{equation}\label{eq:energyrel}
G(c',c')=2(e-V(c))=f(c).
\end{equation}
Then $\widetilde G(\gamma',\gamma')=f\,G(\gamma',\gamma')
=f^{-1}G(c',c')=1$ by \eqref{eq:energyrel}, so $\gamma$ has unit
$\widetilde G-$speed. Substituting \eqref{eq:energyrel} and the
equation of motion into \eqref{eq:jacobikey},
\[
\widetilde\nabla_{\partial_s}\gamma'
=\frac{1}{f^2}\bigl(-\vgrad V\bigr)+\frac{f}{f^3}\vgrad V=0,
\]
so $\gamma$ is a unit speed geodesic of $\widetilde G$.

\ref{jac2} Let $\gamma$ be a unit speed $\widetilde G-$geodesic,
$\widetilde\nabla_{\partial_s}\gamma'=0$ and
$\widetilde G(\gamma',\gamma')\equiv1$, and put $c(t)=\gamma(s(t))$ with
$dt/ds=1/f(\gamma)$. The unit speed condition gives
$1=\widetilde G(\gamma',\gamma')=f\,G(\gamma',\gamma')
=f^{-1}G(c',c')$, that is, $G(c',c')=f=2(e-V(c))$, which is exactly the
energy relation $\Ee_{\mathcal L}(c,c')=e$. Inserting
$G(c',c')=f$ into \eqref{eq:jacobikey} and using
$\widetilde\nabla_{\partial_s}\gamma'=0$,
\[
0=\frac{1}{f^2}\nabla_{\partial_t}c'+\frac{f}{f^3}\vgrad V
=\frac{1}{f^2}\bigl[\nabla_{\partial_t}c'+\vgrad V(c)\bigr],
\]
whence $\nabla_{\partial_t}c'=-\vgrad V(c)$ and, by
Theorem~\ref{thm:mechEL}, $c$ is an extremal of energy $e$.
\end{proof}

\begin{remark}\label{rem:jacobiscope}
There are several points that may need clarification. 
(a) Theorem~\ref{thm:jacobi} works
regardless of the sign of $f=2(e-V)$, as only $f\neq0$ is needed, and this is
the reason why the pseudo-Riemannian (indefinite $G$) and the Riemannian cases
can be treated jointly. In the positive definite, finite dimensional
case with $e>V$ this is the textbook Jacobi metric \cite[\S3.7]{Arn89}. (b)
The set $\{V\neq e\}$ may be disconnected, and then the statement is to be
understood on each connected segment of the curve, on which $f$ has a fixed sign.
When $f<0$ the reparametrization $ds/dt=f$ reverses the orientation, so the
correspondence matches a forward-time solution to a backward-parametrized
geodesic and vice versa. (c) ``Unit speed'' means here
$\widetilde G(\gamma',\gamma')=1$ for the (possibly indefinite) conformal
metric $\widetilde G=fG$. When $\widetilde G$ is indefinite this normalizes
$\gamma$ to be spacelike of unit $\widetilde G-$length, and the multiplication
by the negative factor $f<0$ interchanges the causal labels of $G$ and
$\widetilde G$. No positivity of the speed is asserted. (d) The theorem is a
correspondence between pre-existing smooth solutions. It presupposes
the smooth Christoffel map of $G$ (equivalently, of $\widetilde G$ via
Lemma~\ref{lem:conformal}), and it does not establish by itself existence,
uniqueness, completeness, or a well defined geodesic spray in the weak
Fr\'echet setting, which are topics for a more specialized analytic work.
\end{remark}

\subsection{Gravity as geodesic motion}\label{ssec:GRjacobi}
Let us now particularize to the DeWitt Lagrangian
$\mathcal{L}_{\alpha,\lambda}$, where $V=2\alpha S_\lambda$, and to the
physically relevant energy value $e=0$ (the Hamiltonian constraint fixes
the total energy to zero, by Proposition~\ref{prop:chargeident}\ref{var2}--\ref{var3}
and Corollary~\ref{cor:totalenergy}). Then $e-V=-2\alpha S_\lambda$ and
the Jacobi metric of Theorem~\ref{thm:jacobi} is
\[
\widetilde G=2(e-V)\,G^{-}=-4\alpha\,S_\lambda(g)\,G^{-}=:\mathcal{G}.
\]

\begin{theorem}\label{thm:GRjacobi}
Let $g_t$ be a solution of the Einstein evolution system
\eqref{eq:evolution}--\eqref{eq:constraints} (equivalently, by
Theorem~\textup{\ref{thm:GMtheorem}}, the curve of spatial metrics of an
Einstein spacetime $\widetilde g=\alpha dt^2+g_t$ with
$\rho(\widetilde g)=\lambda\widetilde g$), and suppose $S_\lambda(g_t)\neq0$
for all $t$ in some subinterval $I'\subset I$. Then, reparametrizing by
\[
\frac{ds}{dt}=-4\alpha\,S_\lambda(g_t),
\]
the curve $\gamma(s)=g_{t(s)}$ is a unit speed geodesic of the conformal
DeWitt metric $\mathcal{G}=-4\alpha S_\lambda(g)\,G^{-}$ on the open set
$\{S_\lambda\neq0\}\subset\Mm$,
\[
\mathcal{G}(\gamma',\gamma')\equiv1,\qquad
\widetilde\nabla^{\mathcal{G}}_{\partial_s}\gamma'=0 .
\]
For the Lorentzian vacuum case $\alpha=-1$, $\lambda=0$ this says that the
spatial geometry of a vacuum spacetime traces, up to the
reparametrization $ds/dt=4S(g_t)$, a unit speed geodesic of the metric
$4S(g)\,G^{-}$ on $\{S\neq0\}\subset\Mm$.

Conversely, let $\gamma$ be a unit speed $\mathcal{G}-$geodesic lying in
$\{S_\lambda\neq0\}$, and let $c(t)=\gamma(s(t))$ be its inverse
reparametrization (with $dt/ds=1/f$, $f=-4\alpha S_\lambda$), a
zero-total-energy extremal of $\mathcal{L}_{\alpha,\lambda}$; write
$\dot c:=dc/dt$ for its mechanical velocity, which is the $g'$ of the
evolution equation \eqref{eq:evolution}. If in addition $c$ verifies
\emph{both} constraints\footnote{Notice that the
constraints must be imposed on the mechanical velocity $\dot c=dc/dt$, and
not on $\gamma'=d\gamma/ds$, because both differ by the factor $f$ and, while
the momentum constraint is homogeneous in the velocity, the Hamiltonian
constraint is not, as it combines a quadratic kinetic term with a potential
term which does not depend on the velocity.} \eqref{eq:constraints}, that is, the momentum
constraint $\delta^{-}_{g}\dot c=0$ and the pointwise Hamiltonian constraint
$\Ham_\lambda(c,\dot c)=0$, at its initial data $(c(t_0),\dot c(t_0))$
for one value $t_0$, then $c$ is an Einstein evolution.
\end{theorem}

\begin{proof}
By Proposition~\ref{prop:chargeident}\ref{var1} the solution is an extremal of
$\mathcal{L}_{\alpha,\lambda}$ and, by the Hamiltonian constraint and
Proposition~\ref{prop:chargeident}\ref{var2}, its energy is
$\Ee=\int_M\Ham_\lambda\dv_g=0$. Theorem~\ref{thm:jacobi} with $e=0$,
$V=2\alpha S_\lambda$, $f=-4\alpha S_\lambda$ (nowhere zero on $I'$ by
hypothesis) gives the geodesic statement and $\mathcal{G}(\gamma',\gamma')
\equiv1$, the reparametrization being $ds/dt=f=-4\alpha S_\lambda$.

For the converse, the second half of Theorem~\ref{thm:jacobi} (with
$e=0$) already gives that the inverse reparametrization $c(t)=\gamma(s(t))$
is an extremal of $\mathcal{L}_{\alpha,\lambda}$ with total energy $\Ee=0$,
that is, a solution of the evolution equation \eqref{eq:evolution} with
$\int_M\Ham_\lambda(c,\dot c)\dv=0$. This mechanical
correspondence, however, does not force by itself the pointwise
Hamiltonian constraint, since a time independent density
$\Ham_\lambda\dv_g$ with vanishing integral need not vanish pointwise.
Suppose now that both constraints \eqref{eq:constraints} hold at the initial
data $(c(t_0),\dot c(t_0))$, where $t_0$ is the mechanical instant
corresponding to the chosen geodesic parameter. All the propagation
results are stated for the critical curve in the mechanical parameter
$t$, so we apply them to $c(t)$. Theorem~\ref{thm:momprop} propagates
$\delta^{-}_{g}\dot c=0$ for all $t$, and Corollary~\ref{cor:pointwise}
then makes $\Ham_\lambda\dv_g$ pointwise constant in $t$; as it
vanishes pointwise at $t_0$, it vanishes pointwise for all $t$. Both
constraints of \eqref{eq:constraints} therefore hold along $c$ throughout,
and Corollary~\ref{cor:hamprop} identifies $c$ with an Einstein evolution.
The conclusion applies to $\gamma$ under the reparametrization if
desired, both curves having the same image.
\end{proof}

\begin{remark}\label{rem:degeneracy}
The hypothesis $S_\lambda\neq0$ is really needed. Along a constrained solution, 
the Hamiltonian constraint
$\Ham_\lambda=\tfrac12\ip{k}{k}{g}^{-}+2\alpha\tau_\lambda=0$ integrates
to $\tfrac12 G^{-}(k,k)+2\alpha S_\lambda=0$, so
\[
S_\lambda(g)=-\frac{1}{4\alpha}\,G^{-}(g',g') .
\]
Thus $S_\lambda(g_{t})=0$ if and only if the integrated DeWitt
norm $G^{-}(g',g')$ vanishes, and this can happen in two ways. A
\emph{moment of time symmetry} $g'=k=0$ (for instance the neck of a
recollapsing cosmology, or the turning point of the closed models below) is
one of them. There $S_\lambda=0$, the Jacobi conformal factor
$-4\alpha S_\lambda$ vanishes, $\mathcal{G}$ degenerates and the
reparametrization $ds/dt\to0$ becomes singular. But because $G^{-}$ is
indefinite, the integral
$G^{-}(g',g')=\int_M\ip{g'}{g'}{g}^{-}\dv_g$ can also vanish by
cancellation with $g'\neq0$, so time symmetry is a sufficient, but not a
necessary, source of degeneracy. This is the coordinate free
counterpart of the familiar fact from the finite dimensional
Maupertuis--Jacobi principle \cite[\S3.7]{Arn89} that the Jacobi metric
$2(e-V)G$ degenerates at the boundary of the Hill region $\{V=e\}$, which here
is the locus $\{S_\lambda=0\}$. It is also the reason why we stated
Theorem~\ref{thm:GRjacobi} on the open set $\{S_\lambda\neq0\}$, as on each
connected segment where $S_\lambda\neq0$ the trajectory is reparametrized
as a $\mathcal{G}-$geodesic.
\end{remark}

\section{Homothetic solutions. Milne, de Sitter and the round sphere}
\label{sec:examples}
Let us illustrate the preceding results on the simplest nontrivial curves of metrics,
the homotheties of a fixed Einstein metric. Besides being explicit, these examples show
that all three ingredients (evolution equation, momentum constraint,
Hamiltonian constraint) collapse to elementary scalar ordinary differential
equations, so they can be explicitly calculated, and they exhibit both signatures.

\begin{proposition}[Homothetic reduction]\label{prop:homothetic}
Let $g_0\in\Mm$ be Einstein, $\rho(g_0)=\tfrac{\tau_0}{n}g_0$ with
$\tau_0=\tau(g_0)$ constant, and consider $g_t=c(t)\,g_0$ with
$c(t)>0$. Then:
\begin{enumerate}[label=\textit{\alph*)}]
\item\label{homo1} the momentum constraint $\delta^{-}_{g_t}g'_t=0$ holds
automatically;
\item\label{homo2} the evolution equation \eqref{eq:evolution} reduces to the scalar
second order equation
\begin{equation}\label{eq:creduced}
c''=\frac{4-n}{4}\,\frac{(c')^{2}}{c}
+\frac{\alpha(n-2)\,\tau_0}{n(n-1)}-\alpha\lambda\,c ;
\end{equation}
\item\label{homo3} the Hamiltonian constraint $\Ham_\lambda(g_t,g'_t)=0$ reduces to the
first order relation
\begin{equation}\label{eq:cconstraint}
(c')^{2}=A\,c-B\,c^{2},\mbox{ where }
A:=\frac{4\alpha\,\tau_0}{n(n-1)},\mbox{ and } B:=\frac{4\alpha\lambda}{n};
\end{equation}
\item\label{homo4} on the constraint set \eqref{eq:cconstraint}, the evolution
\eqref{eq:creduced} is equivalent to the \emph{linear} equation
\begin{equation}\label{eq:clinear}
c''=\tfrac12 A-B\,c .
\end{equation}
\end{enumerate}
\end{proposition}

\begin{proof}
Write $k=g'=c'g_0=(c'/c)g_t$, so that $g_t^{-1}k=(c'/c)\,\mathrm{Id}$, and
$\tr_{g_t}k=n\,c'/c$. Hence $k_0=0$ and
$\delta^{-}_{g_t}k=\delta_{g_t}k+d(\tr_{g_t}k)$. Now
$\delta_{g_t}k=\delta_{g_t}((c'/c)g_t)=-d(c'/c)=0$ (as $c'/c$ is a
constant on $M$) and $d(\tr_{g_t}k)=d(nc'/c)=0$, so \ref{homo1} holds.

For \ref{homo2} and \ref{homo3} we use the known behavior of the curvatures
under scalings.
Under $g\mapsto cg$ ($c$ a positive constant) one has $\rho(cg)=\rho(g)$,
$\tau(cg)=c^{-1}\tau(g)$, $\dv_{cg}=c^{n/2}\dv_g$, and
$g_t^{-1}=c^{-1}g_0^{-1}$. Inserting $g=g_t=cg_0$, $k=c'g_0$ into the
explicit evolution equation \eqref{eq:evolutionexplicit}, and using
$kg^{-1}k=(c')^2 g_0 (cg_0)^{-1}g_0=\frac{(c')^2}{c}g_0$,
$\tr_g(k)=nc'/c$, $\ip{k}{k}{g}^{-}=\tr((g^{-1}k)^2)-(\tr g^{-1}k)^2
=n(c'/c)^2-n^2(c'/c)^2=-n(n-1)(c'/c)^2$, $\rho(g)=\tfrac{\tau_0}{n}g_0$,
$\tau(g)=\tau_0/c$, we get (all tensors being proportional to $g_0$):
\[
c''g_0=\frac{(c')^2}{c}g_0-\tfrac12\frac{nc'}{c}c'g_0
+\frac{n(n-1)}{4(n-1)}\frac{(c')^2}{c^2}\,cg_0
+2\alpha\frac{\tau_0}{n}g_0-\frac{\alpha}{n-1}\frac{\tau_0}{c}cg_0
-\alpha\lambda cg_0 .
\]
Dividing by $g_0$ and simplifying the $(c')^2/c$ terms,
$1-\tfrac n2+\tfrac n4=\tfrac{4-2n+n}{4}=\tfrac{4-n}{4}$, and the
potential terms $\frac{2\alpha\tau_0}{n}-\frac{\alpha\tau_0}{n-1}
=\alpha\tau_0\frac{2(n-1)-n}{n(n-1)}=\frac{\alpha(n-2)\tau_0}{n(n-1)}$,
we obtain \eqref{eq:creduced}.

For \ref{homo3}, $\Ham_\lambda=\tfrac12\ip{k}{k}{g}^{-}+2\alpha\tau_\lambda
=-\tfrac{n(n-1)}{2}\frac{(c')^2}{c^2}+2\alpha(\tau_0/c-\lambda(n-1))$;
setting this to zero and multiplying by $-\frac{2c^2}{n(n-1)}$ gives
$(c')^2=\frac{4\alpha\tau_0}{n(n-1)}c-\frac{4\alpha\lambda}{n}c^2=Ac-Bc^2$,
which is \eqref{eq:cconstraint}. 

Finally, for \ref{homo4}, differentiate
\eqref{eq:cconstraint} to get $2c'c''=Ac'-2Bcc'$, so $c''=\tfrac12A-Bc$
wherever $c'\neq0$, and by continuity everywhere. Conversely, one can check that
\eqref{eq:clinear} together with one value of \eqref{eq:cconstraint}
propagates the constraint (the function $(c')^2-Ac+Bc^2$ has vanishing
$t-$derivative by \eqref{eq:clinear}). The equivalence with
\eqref{eq:creduced} on the constraint set follows by substituting
$(c')^2=Ac-Bc^2$ into the $\frac{4-n}{4}(c')^2/c$ term, which becomes
$\frac{4-n}{4}\frac{Ac-Bc^2}{c}=\frac{4-n}{4}(A-Bc)$, and by rewriting the
remaining terms of \eqref{eq:creduced} in terms of $A,B$ through
$\frac{\alpha(n-2)\tau_0}{n(n-1)}=\frac{n-2}{4}A$ and
$\alpha\lambda=\frac{n}{4}B$, so that
\begin{align*}
c''=& \frac{4-n}{4}(A-Bc)+\frac{n-2}{4}A-\frac{n}{4}Bc \\
=& \frac{(4-n)+(n-2)}{4}\,A-\frac{(4-n)+n}{4}\,Bc \\
=& \tfrac12A-Bc,
\end{align*}
which is \eqref{eq:clinear}.
\end{proof}

\begin{remark}
The reduction to the \emph{linear} oscillator (or repulsor) \eqref{eq:clinear}
is the mini-superspace manifestation of Theorem~\ref{thm:GRjacobi}, as the dynamics
collapses to a single scalar degree of freedom $c$, for which the
constrained evolution is elementary. The sign of $B=4\alpha\lambda/n$
determines whether there is oscillation ($B>0$) or exponential behavior ($B<0$), 
while $A$ simply shifts the equilibrium.
\end{remark}

\begin{example}[The Milne universe. $\alpha=-1$, $\lambda=0$, $\tau_0<0$]
\label{ex:milne}
Take $n=3$ and $g_0$ hyperbolic with $\tau_0=-6$ (so that
$\rho(g_0)=-2g_0$). Then $A=\frac{4(-1)(-6)}{3\cdot2}=4$, $B=0$, and
\eqref{eq:clinear} is $c''=2$ with first integral $(c')^2=4c$
\eqref{eq:cconstraint}. For $t>0$ the solution with $c\to0$ as $t\to0^+$
is $c(t)=t^2$, which gives on the region $t>0$ (where $c>0$, so that
$c(t)g_0\in\Mm$)
\[
\widetilde g=-dt^2+t^2 g_0,
\]
the Milne form of (a wedge of) Minkowski space. Indeed,
$\widetilde g$ is flat, and this is the standard hyperbolic foliation of the
interior of the light cone. Here $S_0(g_t)=\int\tau\dv=\tau_0 c^{-1}\cdot
c^{3/2}V_0=\tau_0 c^{1/2}V_0\to0$ as $t\to0^+$, so the Jacobi factor of
Theorem~\ref{thm:GRjacobi} degenerates as $t\to0^+$. The limit $t=0$,
where $c=0$, lies on the boundary of $\Mm$ rather than on the curve, and
it is the vertex of the cone, a coordinate singularity of the Milne slicing.
\end{example}

\begin{example}[de Sitter. $\alpha=-1$, $\lambda>0$, round $S^3$]
\label{ex:desitter}
Take $n=3$, $g_0$ the round metric on $S^3$ with $\tau_0=6$
($\rho(g_0)=2g_0$), and $\lambda>0$. Then
$A=\frac{4(-1)(6)}{6}=-4$, $B=\frac{4(-1)\lambda}{3}=-\tfrac{4\lambda}{3}$,
and \eqref{eq:clinear} is $c''=-2+\tfrac{4\lambda}{3}c$ with
$(c')^2=-4c+\tfrac{4\lambda}{3}c^2$ \eqref{eq:cconstraint}. The solution
bounded below by its turning point is
\[
c(t)=\frac{3}{\lambda}\cosh^{2}\!\Bigl(\sqrt{\tfrac{\lambda}{3}}\,t\Bigr)
=\frac{3}{2\lambda}\Bigl(1+\cosh\sqrt{\tfrac{4\lambda}{3}}\,t\Bigr),
\]
so that $\widetilde g=-dt^2+c(t)g_0$ is the global de Sitter metric with
$\rho=\lambda\widetilde g$. The turning point $t=0$ is a moment of time
symmetry ($c'(0)=0$, that is, $k=0$). There $S_\lambda=0$ and the Jacobi
picture degenerates (Remark~\ref{rem:degeneracy}), in agreement with the
$O(4)$-symmetric bounce of de Sitter space.
\end{example}

\begin{example}[The round $S^{4}$]\label{ex:sphere}
The independence of the signature (Remark~\ref{rem:transportremarks}(c)) is
manifest on taking $\alpha=+1$, $\lambda>0$, $n=3$, $g_0$ round on $S^3$
($\tau_0=6$). Then $A=\frac{4(6)}{6}=4$, $B=\frac{4\lambda}{3}>0$, and
\eqref{eq:clinear} is the \emph{harmonic oscillator}
$c''=2-\tfrac{4\lambda}{3}c$, with
$(c')^2=4c-\tfrac{4\lambda}{3}c^2$. The solution
\[
c(t)=\frac{3}{\lambda}\cos^{2}\!\Bigl(\sqrt{\tfrac{\lambda}{3}}\,t\Bigr)
\]
gives the Riemannian ``spacetime'' $\widetilde g=dt^2+c(t)g_0$, which is
the round metric on $S^{4}$ of radius $\sqrt{3/\lambda}$ (Einstein, with
$\rho=\lambda\widetilde g$), presented as a warped product over an
interval with $S^3$ fibres. On the open interval
$t\in\bigl(-\tfrac{\pi}{2}\sqrt{3/\lambda},\,\tfrac{\pi}{2}\sqrt{3/\lambda}\bigr)$
one has $c>0$ and the curve lies in $\Mm$, and it is time symmetric at the
equator $t=0$ (where $c'=0$, $k=0$). As $t\to\pm\frac{\pi}{2}\sqrt{3/\lambda}$
both $c$ and $c'$ tend to $0$ (from
$c'=-\sqrt{3/\lambda}\,\sin(2\sqrt{\lambda/3}\,t)$, which vanishes together
with $c$ at the poles), so the metric $c(t)g_0$ leaves $\Mm$. The $S^3$
fibre collapses, and the poles are boundary points of the model, not
interior points of the curve of metrics. The smooth collapse is carried by the
$S^3$ radius $a(t)=\sqrt{c(t)}$, whose one-sided derivative is nonzero at
each pole. Thus, the same linear equation \eqref{eq:clinear} interpolates
between the Lorentzian (de Sitter, hyperbolic $c$) and the Riemannian
(sphere, trigonometric $c$) cases through the sign of $\alpha$.
\end{example}



\begin{thebibliography}{99}

\bibitem{ADM62} R.~Arnowitt, S.~Deser, C.W.~Misner,
\emph{The dynamics of general relativity}, in: L.~Witten (ed.),
Gravitation: An Introduction to Current Research, Wiley, New York, 1962,
pp.~227--265; reprinted in Gen. Relativity Gravitation \textbf{40}
(2008), 1997--2027.

\bibitem{Arn89} V.I.~Arnold, \emph{Mathematical Methods of Classical
Mechanics}, 2nd ed., Grad. Texts in Math. \textbf{60}, Springer, New York,
1989.

\bibitem{BHM13} M.~Bauer, P.~Harms, P.W.~Michor,
\emph{Sobolev metrics on the manifold of all Riemannian metrics},
J. Differential Geom. \textbf{94} (2013), no.~2, 187--208.

\bibitem{Bes87} A.L.~Besse, \emph{Einstein Manifolds},
Ergeb. Math. Grenzgeb. (3) \textbf{10}, Springer, Berlin, 1987.

\bibitem{CB09} Y.~Choquet-Bruhat, \emph{General Relativity and the
Einstein Equations}, Oxford Univ. Press, Oxford, 2009.

\bibitem{Cla10} B.~Clarke, \emph{The metric geometry of the manifold of
Riemannian metrics over a closed manifold}, Calc. Var. Partial
Differential Equations \textbf{39} (2010), no.~3--4, 533--545.

\bibitem{Cla13} B.~Clarke, \emph{The completion of the manifold of
Riemannian metrics}, J. Differential Geom. \textbf{93} (2013), no.~2,
203--268.

\bibitem{DeW67} B.S.~DeWitt, \emph{Quantum theory of gravity. I. The
canonical theory}, Phys. Rev. \textbf{160} (1967), 1113--1148.

\bibitem{Ebi70} D.G.~Ebin, \emph{The manifold of Riemannian metrics},
in: Global Analysis (Berkeley, 1968), Proc. Sympos. Pure Math. \textbf{15},
Amer. Math. Soc., Providence, 1970, pp.~11--40.

\bibitem{FG89} D.S.~Freed, D.~Groisser, \emph{The basic geometry of the
manifold of Riemannian metrics and of its quotient by the diffeomorphism
group}, Michigan Math. J. \textbf{36} (1989), no.~3, 323--344.

\bibitem{FM72a} A.E.~Fischer, J.E.~Marsden, \emph{The Einstein evolution
equations as a first-order quasi-linear symmetric hyperbolic system, I},
Comm. Math. Phys. \textbf{28} (1972), 1--38.

\bibitem{FM72b} A.E.~Fischer, J.E.~Marsden, \emph{The Einstein equations
of evolution --- a geometric approach}, J. Math. Phys. \textbf{13} (1972),
546--568.

\bibitem{GMi91} O.~Gil-Medrano, P.W.~Michor, \emph{The Riemannian manifold
of all Riemannian metrics}, Quart. J. Math. Oxford Ser. (2) \textbf{42}
(1991), no.~166, 183--202.

\bibitem{GM96} O.~Gil-Medrano, \emph{A Riemannian generalization of a
result of DeWitt concerning Ricci-flat Lorentz metrics}, J. Math. Phys.
\textbf{37} (1996), no.~8, 4017--4024.


\bibitem{GM01} O.~Gil-Medrano, \emph{Relationship between volume and
energy of vector fields}, Differential Geom. Appl. \textbf{15} (2001),
137--152.

\bibitem{Giu09} D.~Giulini, \emph{The superspace of geometrodynamics},
Gen. Relativity Gravitation \textbf{41} (2009), no.~4, 785--815.

\bibitem{Ham82} R.S.~Hamilton, \emph{The inverse function theorem of
Nash and Moser}, Bull. Amer. Math. Soc. (N.S.) \textbf{7} (1982),
no.~1, 65--222.

\bibitem{KM97} A.~Kriegl, P.W.~Michor, \emph{The Convenient Setting of
Global Analysis}, Math. Surveys Monogr. \textbf{53}, Amer. Math. Soc.,
Providence, 1997.

\bibitem{MR99} J.E.~Marsden, T.S.~Ratiu, \emph{Introduction to Mechanics
and Symmetry}, 2nd ed., Texts Appl. Math. \textbf{17}, Springer,
New York, 1999.

\bibitem{Mic20} P.W.~Michor, \emph{Manifolds of mappings and shapes},
in: The Legacy of Bernhard Riemann after One Hundred and Fifty Years,
Vol.~II, Adv. Lect. Math. \textbf{35.2}, Int. Press, Somerville, MA,
2016, pp.~459--486.

\bibitem{Mis57} C.W.~Misner, \emph{Feynman quantization of general
relativity}, Rev. Modern Phys. \textbf{29} (1957), 497--509.

\bibitem{Rin09} H.~Ringstr\"om, \emph{The Cauchy Problem in General
Relativity}, ESI Lect. Math. Phys., Eur. Math. Soc., Z\"urich, 2009.

\bibitem{Tre06} F.~Tr\`eves, \emph{Topological Vector Spaces,
Distributions and Kernels}, Dover, Mineola, 2006 (reprint of the 1967
Academic Press edition).

\bibitem{Val09} J.A.~Vallejo, \emph{Euler--Lagrange equations for
functionals defined on Fr\'echet manifolds}, J. Nonlinear Math. Phys.
\textbf{16} (2009), no.~4, 443--454.

\bibitem{Whe64} J.A.~Wheeler, \emph{Geometrodynamics and the issue of the
final state}, in: Relativity, Groups and Topology (Les Houches, 1963),
Gordon and Breach, New York, 1964, pp.~315--520.

\end{thebibliography}
\end{document}